\documentclass{article}

\usepackage{amsmath,amsthm,amssymb}
\usepackage{enumitem}
\usepackage{mathrsfs}

\usepackage[a4paper, total={6in, 8in}]{geometry}

\usepackage[utf8]{inputenc}
\usepackage[english]{babel}

\numberwithin{equation}{section}
\newtheorem{theorem}{Theorem}[section]
\newtheorem{theoremA}{Theorem}

\newtheorem{proposition}[theorem]{Proposition}
\newtheorem{lemma}[theorem]{Lemma}
\newtheorem{corollary}[theorem]{Corollary}

\theoremstyle{definition}

\newtheorem{remark}[theorem]{Remark}

\def\N{\mathbb N}
\def\R{\mathbb R}
\def\C{\mathbb C}

\title{Weighted norm inequalities for generalized Fourier-type transforms and applications}	
\author{Alberto Debernardi}
\date{}

\begin{document}

	\maketitle

	{\small
	\noindent \textbf{Abstract}: We obtain necessary and sufficient conditions on weights for the generalized Fourier-type transforms to be bounded between weighted $L^p-L^q$ spaces. As an important example, we investigate transforms with kernel of power type, as for instance the sine, Hankel or $\mathscr{H}_\alpha$ transforms. The obtained necessary and sufficient conditions are given in terms of weights, but not in terms of their decreasing rearrangements, as in several previous investigations.\newline \newline
	\textbf{AMS 2010 Primary subject classification}: 42A38. Secondary:  26D15, 44A15.\newline \newline
	This research was partially funded by the CERCA Programme of the Generalitat de Catalunya, Centre
	de Recerca Matem\`atica, Fundaci\'o Ferran Sunyer i Balaguer from Institut d'Estudis Catalans, and the MTM2014--59174--P grant.
	}

	\section{Introduction}
	\subsection{Weighted norm inequalities for the Fourier transform}
	Given an integral operator $T$ and $1\leq p,q\leq \infty$, determining necessary and sufficient conditions on pairs of nonnegative locally integrable functions $u,v:\R^n\to \R_+$ (also called \textit{weights}) in order for the inequality
	\begin{equation}\label{eq1}
	\bigg(\int_{\R^n} u(y)|Tf(y)|^q\, dy\bigg)^{1/q}\leq C_{T,n,p,q}\bigg(\int_{\R^n} v(x)|f(x)|^p\, dx\bigg)^{1/p},
	\end{equation}
	to be satisfied for every measurable $f$ (with $C$ independent of $f$) is an important problem in analysis. One of the main examples of such transform $T$ is the Fourier transform
	$$
	\widehat{f}(y)=\int_{\R^n} f(x)e^{ix\cdot y}\, dx,
	$$
for which \eqref{eq1} is rewritten as
	\begin{equation}\label{wnifourier}
\bigg(\int_{\R^n}u(y)|\widehat{f}(y)|^q\, dy\bigg)^{1/q}\leq C_{n,p,q}\bigg(\int_{\R^n} v(x) |f(x)|^p\, dx\bigg)^{1/p}, \quad 1\leq p,q\leq\infty.
	\end{equation}
	
	Important examples of applications of the above inequalities are the study of uncertainty principle relations (cf. \cite{Be}) or restriction inequalities \cite{Fef,tomas,DCGTrestr}. Inequality \eqref{wnifourier} and its variants have been extensively studied, see  \cite{AH,BeAnnals,BHmeasure,BH,JurkatSampson,MuckenhouptWeighted} and the references therein.
	
	The following is well known, and was proved independently by Heinig \cite{heinigWeighted}, Jurkat-Sampson \cite{JurkatSampson} and Muckenhoupt \cite{MuckenhouptNote,MuckenhouptWeighted} in the 1980s (we take $n=1$): \textit{if $u,v$ are even, nonincreasing and nondecreasing as functions of $|x|$ respectively, then inequality \eqref{wnifourier} holds with $1<p\leq q<\infty$ if and only if
	$$
	\sup_{r>0}	\bigg(\int_0^{1/r} u(x)\, dx\bigg)^{1/q}\bigg(\int_0^{r}v(x)^{1-p'}\, dx\bigg)^{1/p'}\leq C.
	$$
}

	Typical examples of weights $u,v$ are power functions. In this case, important examples of \eqref{wnifourier} are the Hausdorff-Young inequality \cite{babenko,BeAnnals}
	$$
	\bigg(\int_{\R^n}|\widehat{f}(x)|^{p'}\, dx\bigg)^{1/p'}\leq C\bigg(\int_{\R^n}|f(x)|^p\, dx\bigg)^{1/p},  \qquad 1\leq p\leq 2,
	$$
	and the Hardy-Littlewood inequality 
	$$
	\bigg(\int_{\R^n}|x|^{p-2}|\widehat{f}(x)|^p\, dx\bigg)^{1/p}\leq C\bigg(\int_{\R^n}|f(x)|^p\, dx\bigg)^{1/p},  \qquad 1<p\leq 2,
	$$
see \cite{SWfourier,titchmarsh}. More generally, if $u(x)=|x|^{-\beta q}$ and $v(x)=|x|^{\gamma p}$, inequality \eqref{wnifourier} is known as the classical Pitt's inequality, and it holds if and only if
	\begin{equation}
	\label{fouriersharprange}
	\beta=\gamma+n\bigg(\frac{1}{q}-\frac{1}{p'}\bigg), \qquad  \max\bigg\{ 0,n\bigg(\frac{1}{q}-\frac{1}{p'}\bigg)\bigg\}\leq \beta<\frac{n}{q}.
		\end{equation}
	
	Another interesting problem is to study whether the sharp range for $\beta$ given above can be extended when considering regularity conditions on $f$ (cf. \cite{GLTBoas,LTparis,SaWh}). For instance, if $f$ is a radial function defined on $\R^n$  (i.e., $f(x)=f_0(|x|)$), inequality \eqref{wnifourier} holds if and only if
	$$
		\beta=\gamma+n\bigg(\frac{1}{q}-\frac{1}{p'}\bigg), \qquad \frac{n}{q}-\frac{n-1}{2}+\max\bigg\{0,\frac{1}{q}-\frac{1}{p'}\bigg\}\leq \beta <\frac{n}{q}.
	$$
	If additionally $f_0$ satisfies general monotonicity conditions (see Section~\ref{sectionGM}), the latter range can be improved to
	$$
	\frac{n}{q}-\frac{n+1}{2}<\beta<\frac{n}{q}.
	$$
	Such monotonicity assumption sometimes allows us to weaken the sufficient conditions the weights $u,v$ should satisfy to guarantee that \eqref{eq1} holds, and in fact it plays a key role in Section~\ref{sectionGM}.
	
	In what follows all integral transforms we consider are one-dimensional and defined on $\R_+$. We will denote, for any function $f$, any weight $v$, and any $0<p<\infty$,
	$$
	\Vert f\Vert_{p}:=\bigg(\int_{0}^{\infty}|f(x)|^p\, dx\bigg)^{1/p}, \qquad \Vert f\Vert_{p,v}:=\Vert v^{1/p} f\Vert_p=\bigg(\int_{0}^{\infty} v(x)|f(x)|^p\, dx\bigg)^{1/p}.
	$$
	
	\subsection{Integral transforms of Fourier type}
	Following \cite{DLT}, for a complex-valued function $f$ we denote 
	\begin{equation}
	\label{transform}
		Ff(y)=\int_0^\infty s(x) f(x)  K(x,y)\, dx, \qquad y>0, 
	\end{equation}
	where $K$ is a continuous kernel and $s$ is a nonnegative nondecreasing function such that 
	\begin{equation}
		\label{doubling}
		s(x)\lesssim s(2x), \qquad x>0.
	\end{equation}
	 Furthermore, we assume that there  exists a nonnegative nondecreasing function $w$ satisfying
	\begin{equation}
	\label{inverse}
	s(x)w(1/x)\asymp 1, \qquad x>0,
	\end{equation}
	and such that the estimate
	\begin{equation}
	\label{kernelest}
	|K(x,y)|\lesssim \min\big\{1, (s(x)w(y))^{-1/2}\big\}, \qquad x,y>0,
	\end{equation}
	holds.  Moreover, we suppose that $f,s\in L^1_{\textrm{loc}}$, and
	\begin{equation}\label{integrabilitycond}
	\int_0^1 s(x)|f(x)|\, dx + \int_1^\infty s(x)^{1/2}|f(x)|\, dx<\infty,
	\end{equation}
	so that $Ff(y)$ is pointwise defined on $(0,\infty)$. Note that in this case the estimate
	\begin{equation}
	\label{pointwiseest}
	|Ff(y)|\lesssim \int_0^{1/y}s(x)|f(x)|\, dx + w(y)^{-1/2}\int_{1/y}^\infty s(x)^{1/2}|f(x)|\, dx
	\end{equation}
	holds. We remark that the weight $s$ could be incorporated into the kernel $K$; however, it is worth considering it separately, as it appears as one of the two factors in the estimate \eqref{kernelest}. Another reason to separate $s$ from $K$ is to stay close to the framework of the so-called  \emph{Fourier-type} transforms, also referred to as $F$-transforms (see \cite{HaTi,titchmarsh,wat} and the recent paper \cite{DLT}), i.e., those satisfying \eqref{doubling}--\eqref{kernelest}, and for which there exists $C>0$ such that if $f\in L^2_s$, (or in other words, $\Vert f\Vert_{2,s}<\infty$), then
\begin{equation}
\label{besselineq}
\Vert Ff\Vert_{2,w}\leq C\Vert f\Vert_{2,s}.
\end{equation}
The latter is known as \emph{weighted Bessel's inequality}. Classical examples of Fourier-type transforms are the sine and cosine transforms, or the Hankel transform, which is introduced in the next subsection.

Weighted norm inequalities of transforms with such kernels have been studied in detail in \cite{DLT}, where the authors obtained  the following sufficient conditions that guarantee inequality \eqref{eq1} for $F$-transforms: let us denote by $u^*$ the decreasing rearrangement of $u$, and by $v_*=[(1/v)^*]^{-1}$. For any $1\leq a\leq \infty$, we denote $a'=a/(a-1)$ the dual exponent of $a$.
\begin{theoremA}\label{thmSUFF-DLT}
	Let $1<p\leq q<\infty$, $1<a\leq 2$, $(p,q,a)\neq (2,2,2)$. Let $u,v$ be weights satisfying
	\begin{align}
	\sup_{r>0}\bigg(\int_0^{1/r}u^*(y)\, dy\bigg)^{1/q}\bigg(\int_0^r v_*(x)^{1-p'}\, dx\bigg)^{1/p'}&\leq C,\label{suffDLTcond1}\\
	\sup_{r>0}\bigg(\int_{1/r}^\infty y^{-q/a'}u^*(y)\, dy\bigg)^{1/q}\bigg(\int_r^\infty x^{-p/a'}v_*(x)^{1-p'}\, dx\bigg)^{1/p'} &\leq C.\label{suffDLTcond2}
	\end{align}
	Then, the following inequality holds:
	\begin{equation}
			\label{eqpittDLT}
		\Vert w^{1/a'}Ff\Vert_{q,u} \lesssim \Vert s^{1/a}f\Vert_{p,v}.
	\end{equation}
	If $(p,q,a)=(2,2,2)$ and $u,v$ are weights satisfying
	$$
	\sup_{r>0}\bigg(\int_0^{1/r}u^*(y)\, dy\bigg)\bigg(\int_0^r v_*(x)^{1-p'}\, dx\bigg)\leq C, 
	$$
	the following inequality holds:
	$$
	\Vert w^{1/2}Ff\Vert_{2,u} \lesssim \Vert s^{1/2}f\Vert_{2,v}.
	$$
\end{theoremA}

In this paper we deal with transforms of the form \eqref{transform} for which estimate \eqref{kernelest} holds. These are more general than Fourier-type transforms, since  conditions \eqref{doubling}--\eqref{integrabilitycond} do not imply \eqref{besselineq} in general.

\subsection{Integral transforms with power-type kernel}

We also define the transforms with \textit{kernels of power type} (or power-type kernels) as those of the form
	\begin{equation}
	\label{powertrans}
	Ff(y)=y^{c_0}\int_0^\infty x^{b_0} f(x) K(x,y)\, dx, 
	\end{equation}
	where 
	\begin{equation}
	\label{kernelestcor}
	|K(x,y)|\lesssim \min \{x^{b_1}y^{c_1}, x^{b_2}y^{c_2}\},
	\end{equation}
	with $b_j,c_j \in \R$ for $0\leq j\leq 2$. It is clear that every transform of the form \eqref{transform} satisfying the estimate \eqref{kernelest} with $s(x)=x^{\delta}$, $\delta\in \R$, is a  transform with power-type kernel, but the converse is not true.
	
	 Here, in order for $Ff(y)$	to be well defined we assume 
	$$
	\int_0^1 x^{b_0+b_1}|f(x)|\, dx+\int_1^\infty x^{b_0+b_2}|f(x)|\, dx<\infty.
	$$
	
	Note that in general the kernels of the type $K(x,y)\asymp \min\{x^{b_1}y^{c_1}, x^{b_2}y^{c_2}\}$ differ from the kernels satisfying \textit{Oinarov's condition} \cite{oinarov}, i.e., for some $d>0$,
	\begin{equation}
	\label{oinarovcond}
	d^{-1}(K(t,u)+K(u,v))\leq K(t,v)\leq d(K(t,u)+K(u,v)), \qquad 0<v\leq u\leq t<\infty.
	\end{equation}
	In the case $b_j=c_j=0$, $j=1,2$, it is clear that $K(x,y)\asymp 1$ implies \eqref{oinarovcond}. However, this case is of no interest for us, as our main result for transforms with power-type kernel is not applicable (see Corollary~\ref{corapplications} below). However, if 
	$$
	K(x,y)\asymp \begin{cases} 1, &\text{if }xy\leq 1,\\
	(xy)^{-\delta},& \text{if }xy>1,
	\end{cases}
	$$
	with $\delta>0$, and for $N$ big enough we set $t=N^{\alpha}$, $u=N^{\beta}$, $v=N^{-(\alpha+\beta)/2}$, with $\alpha>\beta>0$, then \eqref{oinarovcond} reads as
	$$
	1\lesssim N^{\delta\frac{\alpha-\beta}{2}}\lesssim 1,
	$$
	which is clearly not true.

\subsection{The Hankel and $\mathscr{H}_\alpha$ transforms}
Two important examples of transforms illustrating the operators from the previous subsection are the Hankel and the $\mathscr{H}_\alpha$ transforms. The former is defined as (cf. \cite{titchmarsh})
\begin{equation}
\label{defHankel}
H_\alpha f(y)=\int_0^\infty x^{2\alpha+1} f(x) j_\alpha(xy)\, dx,
\end{equation}
where $j_\alpha$ is the normalized Bessel function of order $\alpha$, given by the series
 \begin{equation}
 \label{eqSeriesBessel}
 j_\alpha(x)=\Gamma(\alpha+1) \sum_{k=0}^\infty \frac{(-1)^k(x/2)^{2k}}{k!\Gamma(\alpha+k+1)}.
 \end{equation} 
We also mention the identity $j_\alpha(x)=\Gamma(\alpha+1)(x/2)^{-\alpha}J_\alpha(x)$, where $J_\alpha$ is the Bessel function of the first kind of order $\alpha$. The function $j_\alpha(xy)$ satisfies the estimate
 \begin{equation}
 \label{besselest}
 |j_\alpha(xy)|\lesssim \min\big\{1, (xy)^{-\alpha-1/2}\big\}.
 \end{equation}

It holds that $H_\alpha$ is a Fourier-type transform, but also a  transform with power-type kernel. The Hankel transform of order $\alpha=n/2-1$ arises as the Fourier transform of radial functions in $\R^n$, see \cite{SWfourier} (in fact, the cosine transform is nothing more than the Hankel transform of order $\alpha=-1/2$).

In relation with the Hankel transform, we have the $\mathscr{H}_\alpha$ transform, defined as
\begin{equation}
\mathscr{H}_\alpha(y)=\int_0^\infty (xy)^{1/2} f(x)\mathbf{H}_\alpha(xy)\, dx, \qquad \alpha>-1/2 \label{defHtrans},
\end{equation}
see \cite{RonCan,titchmarsh}. Here $\mathbf{H}_\alpha$ is the Struve function of order $\alpha$ \cite{EMOT,watBesselfn}, given by the series
\begin{equation}\label{Hseries}
\mathbf{H}_\alpha(x)=\bigg(\frac{x}{2}\bigg)^{\alpha+1}\sum_{k=0}^\infty \frac{(-1)^k(x/2)^{2k}}{\Gamma(k+3/2)\Gamma(k+\alpha+3/2)}.
\end{equation}
The function $\mathbf{H}_\alpha$ is continuous and satisfies the estimate 
 \begin{equation}
 \label{Hest}
 |\mathbf{H}_\alpha(x)|
 \lesssim \begin{cases}\min\{x^{\alpha+1},x^{-1/2}\}, &\alpha<1/2,\\
 \min\{x^{\alpha+1},x^{\alpha-1}\}, &\alpha\geq 1/2.
 \end{cases}
 \end{equation}
Moreover,  $\mathbf{H}_\alpha$ is related to the Bessel function of the first kind $J_\alpha$ in the following way: $\mathbf{H}_\alpha$ is the solution of the non-homogeneous Bessel differential equation
\begin{equation}
\label{EDO}
x^2\frac{d^2f}{dx^2}+x\frac{df}{dx}+(x^2-\alpha^2)f=\frac{4(x/2)^{\alpha+1}}{\sqrt{\pi}\Gamma(\alpha+1/2)},
\end{equation}
whilst $J_\alpha$ is the solution of the homogeneous differential equation corresponding to \eqref{EDO} that is bounded at the origin for nonnegative $\alpha$.  

The operator $\mathscr{H}_\alpha$ corresponds to a transform with power-type kernel, but if we write it in the form \eqref{transform}, condition \eqref{kernelest} does not hold in general.

We remark that the $\mathscr{H}_\alpha$ transform can be defined for a wider range of $\alpha$ than $\alpha>-1/2$ (see \cite{RonCan} for more details), but for our purpose we need to restrict ourselves to the indicated range.

 Further basic properties of the kernels $j_\alpha$ and $\mathbf{H}_\alpha$ are discussed in Subsection~\ref{ss2.1}.

\subsection{Main results and outline}\label{ssoutline}
The aim of this work is to give simple necessary and sufficient conditions on two different kinds of integral transforms for the weighted norm inequality \eqref{eqpittDLT} to hold. In more detail, we deal with transforms of the form \eqref{transform} whose kernel satisfies \eqref{kernelest} (which generalize the Fourier-type transforms), and those of the form \eqref{powertrans} with power-type kernel. We emphasize that if $F$ as defined in \eqref{transform} is such that $s(x)=x^\delta$ with $\delta\in \R$, and satisfies \eqref{kernelest}, then $F$ has a power-type kernel. Our main tool is Hardy's inequality; this allows us to obtain those conditions written in terms of integrals of the weights $u,v$ instead of their decreasing rearrangements, as in many previous articles within the scope of this topic.
 
In Section~\ref{s2} we list some properties of the normalized Bessel function and the Struve function (the kernels of the Hankel and $\mathscr{H}_\alpha$ transforms, respectively). We also prove an auxiliary lemma related to the antiderivative of the Struve function.

In Section \ref{S3} we study the transforms \eqref{transform} for which the estimate \eqref{kernelest} holds. Note that in contrast with the $F$-type transforms, we do not require that properties \eqref{doubling}, \eqref{inverse} nor Bessel's inequality \eqref{besselineq} hold; as already mentioned, such transforms are  more general than $F$-transforms. The statement yielding sufficient conditions for \eqref{eqpittDLT} to hold reads as follows:
\begin{theorem}\label{suffFtransforms}
	Let $1<p\leq q<\infty$, $1\leq a\leq \infty$, and $u,v$ be nonnegative. Assume there exists $C>0$ such that
	\begin{align}
	\label{sufficiencycond1}
	\sup_{r>0}\bigg(\int_0^{1/r} u(y)w(y)^{q/a'}\, dy\bigg)^{1/q}\bigg( \int_0^r v(x)^{1-p'} s(x)^{p'/a'}\, dx\bigg)^{1/p'}&\leq C,\\
	\sup_{r>0}\bigg(\int_{1/r}^\infty u(y) w(y)^{q(1/a'-1/2)}\, dy\bigg)^{1/q}\bigg(\int_r^{\infty}v(x)^{1-p'} s(x)^{p'(1/a'-1/2)}\, dx\bigg)^{1/p'}&\leq C.\label{sufficiencycond2}
	\end{align}
	Then the weighted norm inequality
	\begin{equation}
	\label{wni}
		\Vert w^{1/a'}Ff\Vert_{q,u} \lesssim \Vert s^{1/a}f\Vert_{p,v}
	\end{equation}  holds for every measurable $f$.
\end{theorem}
\begin{remark}
	Theorem~\ref{suffFtransforms} (and every other assertion in the sequel) is stated for all measurable $f$, although the interesting case is when $\Vert s^{1/a}f\Vert_{p,v}<\infty$. However,  if such a norm is not finite, inequality \eqref{wni} trivially holds, which makes the assertion true for any measurable $f$ (and the same applies for further results).
\end{remark}
	Note that the sufficient conditions of Theorem~\ref{suffFtransforms} depend both on the parameter $a$ and on the weights $s$ and $w$. However, conditions \eqref{suffDLTcond1} and \eqref{suffDLTcond2} from Theorem~\ref{thmSUFF-DLT} do not depend on the weights $s,w$, but only on the parameter $a$. In order to prove Theorem~\ref{thmSUFF-DLT}, in \cite{DLT} the authors make use of Calder\'on's inequality (see also \cite{calderon})
		\begin{equation}
		\label{calderonineq}
		(Tf)^*(y)\lesssim \int_0^{1/y} f^*(x)\, dx+y^{-1/a'}\int_{1/y}^\infty x^{-1/a'}f^*(x)\, dx, \qquad 1<a\leq 2,
		\end{equation}
	applicable to transforms $T$ of type $(1,\infty)$ and $(a,a')$ for all $1<a\leq 2$. If $F$ is a Fourier-type transform, then $Tf=w^{1/a'}F(s^{-1/a}f)$ is of type $(1,\infty)$ and $(a,a')$ for $1<a\leq 2$ (cf. \cite[Lemma 2.1]{DLT}). Thus, the weights $s,w$ appear inside the norm \eqref{eqpittDLT}, but not in conditions \eqref{suffDLTcond1} and \eqref{suffDLTcond2}, so that the appearance of $a$ is essential in the approach of \cite{DLT}.
	
	Note that we can consider the weights $\overline{u}=w^{q/a'}u$ and $\overline{v}=s^{p/a}v$ in place of $u,v$ respectively, so that the parameter $a$ can be omitted. However, we prefer to keep it in the formulation of our results, so that we stay close to the framework of \cite{DLT}. 
	
	Using a so-called ``gluing lemma'' (see \cite{GKP}), it is possible to write conditions \eqref{sufficiencycond1} and \eqref{sufficiencycond2} as only one that is equivalent to the simultaneous fulfilment of these. However, this is only applicable to transforms whose weights satisfy \eqref{inverse}, such as Fourier-type transforms. In order to apply the mentioned gluing lemma, we also need to restrict ourselves to the case $a=1$.
	\begin{corollary}\label{CORgluing}
		Let $1<p\leq q<\infty$ and $u,v$ be nonnegative. Assume that \eqref{inverse} holds and  that there exists $C>0$ such that
		\begin{align}
		\sup_{r>0}\bigg[ \bigg( \int_0^r v(x)^{1-p'}\, dx +& s(r)^{p'/2}\int_r^\infty v(x)^{1-p'}s(x)^{-p'/2}\, dx\bigg)^{1/p'}\nonumber \\
		&\times \bigg( w(1/r)^{q/2}\int_{1/r}^\infty u(y)w(y)^{-q/2}\, dy + \int_0^{1/r} u(y)\, dy  \bigg)^{1/q}\bigg]\leq C, \label{EQglued}
		\end{align}
		Then the weighted norm inequality $\Vert Ff\Vert_{q,u} \lesssim \Vert sf\Vert_{p,v}$ holds for every measurable $f$.
	\end{corollary}
	
Although Theorem~\ref{suffFtransforms} can be applied to a larger number of operators than just the Fourier-type transforms (see Theorem~\ref{thmSUFF-DLT}), we see that it has some limitations. For instance, if $s(x)\asymp w(x)\asymp 1$, it readily follows that we can get no sufficient conditions whenever $u,v$ are power weights, since \eqref{sufficiencycond1} and \eqref{sufficiencycond2} cannot hold simultaneously. This already excludes the classical Fourier transform or the cosine transform from the scope of Theorem~\ref{suffFtransforms}.

Also note that whenever $s$ and $w$ are increasing, \eqref{suffDLTcond1} always implies \eqref{sufficiencycond1}, by Hardy-Littlewood rearrangement inequality (cf. \cite[Ch. II]{BSbook}), which for our purpose can be stated as $
\int_0^{t} u(x)\, dx\leq \int_0^{t} u^*(x)\, dx$ for all $t>0$ and measurable $u$. Moreover, in \cite{DLT} the authors prove that condition \eqref{suffDLTcond2} is redundant in the cases $a'<\max\{q,p'\}$ or $a=p=q=2$, by showing that in these cases \eqref{suffDLTcond1} implies \eqref{suffDLTcond2}.

Theorem~\ref{suffFtransforms} is sharp in general, as shown by considering any transform with kernel satisfying $K(x,y)\asymp\min\big\{1,(s(x)w(y))^{-1/2}\big\}$ (Theorem~\ref{thmnecessity1}). In this case, we can write the following:
\begin{corollary}\label{coriff}
	Let the kernel $K$ from \eqref{transform} satisfy $K(x,y)\asymp\min\big\{1,(s(x)w(y))^{-1/2}\big\}$. Then the inequality $\Vert w^{1/a'}Ff\Vert_{q,u} \lesssim \Vert s^{1/a}f\Vert_{p,v}$ holds for every measurable $f$ if and only if \eqref{sufficiencycond1} and \eqref{sufficiencycond2} are satisfied.
\end{corollary}
As an example of a kernel satisfying the hypotheses of the latter,  consider 
$$
K(x,y)=\begin{cases}
1,& \text{if }xy\leq 1,\\
(s(x)w(y))^{-1/2},&\text{if }xy>1,
\end{cases}
$$
so that
$$
Ff(y)=\int_0^{1/y}s(x)f(x)\, dx +w(y)^{-1/2}\int_{1/y}^\infty s(x)^{1/2}f(x)\, dx.
$$
Furthermore, note that it follows from Theorem~\ref{thmnecessity1} that condition \eqref{sufficiencycond1} is necessary for \eqref{eqpittDLT} to hold (with $a=1$) in the case of the Laplace transform (for which $K(x,y)=e^{-xy}$ and $s\equiv w\equiv 1$), as shown by Bloom in \cite{BloomLaplace}.

Sections~\ref{S4}--\ref{sectionGM} are devoted to the study of weighted norm inequalities for transforms with power-type kernel. We start by considering transforms of the form \eqref{transform} assuming that $s(x)=x^{\delta}$ with $\delta>0$, and taking weights $u,v$ that are piecewise power functions, that is, if for any real numbers 
$\alpha_1,\alpha_2$ we denote $\overline{\alpha}=(\alpha_1,\alpha_2)$, $\overline{\alpha}'=(\alpha_2,\alpha_1)$ and
$$
x^{\overline{\alpha}}:=\begin{cases}
x^{\alpha_1}, &x\leq 1, \\
x^{\alpha_2}, &x>1,
\end{cases} 
$$
then our weights have the form
\begin{equation}
\label{powerweights}
u(x)=x^{-\overline{\beta}'q},\qquad 
v(x)=x^{\overline{\gamma}p},
\end{equation}
with $\beta_i,\gamma_i \in \R$, $i=1,2$. First, we rewrite Theorem~\ref{suffFtransforms} for power weights as follows:
\begin{theorem}\label{theoremsuffpowerweights}
	Let $\beta_1-\gamma_1=\beta_2-\gamma_2$. Let $F$ be of the form \eqref{powertrans} with kernel satisfying \eqref{kernelestcor}. Assume that $u,v$ are of the form \eqref{powerweights}, and that $s(x)=x^\delta$ with $\delta> 0$. If
	\begin{equation}\label{sufpowers}
	\beta_i=\gamma_i+\frac{1}{q}-\frac{1}{p'}, \qquad i=1,2,
	\end{equation}
	with
	\begin{equation}
	\label{sufpowers2}
	\frac{1}{q}-\frac{\delta}{2}<\beta_i<\frac{1}{q}, \qquad i=1,2,
	\end{equation}
	then the inequality 
	\begin{equation}
	\label{pittpower}
	\big\Vert x^{-\overline{\beta}'} Ff\big\Vert_{q}\leq C\big\Vert x^{\overline{\gamma}+\delta}f\big\Vert_p
	\end{equation}
	holds for any measurable $f$.
\end{theorem}

Taking $\beta_1=\beta_2=\beta$ and $\gamma_1=\gamma_2=\gamma$ in Theorem~\ref{theoremsuffpowerweights}, we derive the following corollary for transforms \eqref{powertrans} with kernels of power type.

	\begin{corollary}\label{corapplications}
	Let $1<p\leq q<\infty$, and let $F$ be of the form \eqref{powertrans} with kernel satisfying \eqref{kernelestcor} and such that $b_1-b_2=c_1-c_2>0$. Then, the inequality 
	\begin{equation}
	\label{generalweightednorm}
	\Vert y^{-\beta} Ff\Vert_q\lesssim \Vert x^{\gamma} f\Vert_p
	\end{equation}
holds 	with
	\begin{equation}
	\label{generalweightedconds}
	\beta=\gamma +c_0-b_0+c_1-b_1+\frac{1}{q}-\frac{1}{p'}, \qquad \frac{1}{q}+c_0+c_2 <\beta<\frac{1}{q}+c_0+c_1.
	\end{equation}
\end{corollary}

Here, additionally to the Fourier-type transforms with $s(x)=x^\delta$, $\delta>0$, (Hankel with $\alpha>-1/2$), one can consider any kind of transform as long as its kernel satisfies upper estimates given by power functions, as for instance the sine or $\mathscr{H}_\alpha$ transforms. We remark that although the sine transform is not of Fourier-type itself, it can be written as a weighted Hankel integral.

Similarly as for Theorem~\ref{thmSUFF-DLT}, one can prove that \eqref{generalweightedconds} is necessary for \eqref{generalweightednorm} to hold if $K(x,y)\asymp \min\{x^{b_1}y^{c_1},x^{b_2}y^{c_2}\}$ (see Theorem~\ref{ThmNecessitypowers}).  In particular, the $\mathscr{H}_\alpha$ transform with $\alpha>1/2$ has a kernel satisfying such estimate (see Remark~\ref{remarkHbigalpha} below).

 We list in Section~\ref{S4} the sufficient conditions for \eqref{generalweightednorm} to hold that are derived from Corollary~\ref{corapplications} for each of the aforementioned transforms, as well as the already known necessary and/or sufficient conditions.

In Section~\ref{SectionVanishingMoments} we give an application of Corollary~\ref{corapplications}. We study inequality \eqref{generalweightednorm} for transforms with the kernel represented by  power series
\begin{equation*}
K(x,y)=x^{b_1}y^{c_1}\sum_{m=0}^\infty a_m (xy)^{km}, \qquad k\in \N, \qquad a_m\in \C, \qquad b_1,c_1\in \R, \qquad x,y>0,
\end{equation*}
under certain assumptions. Following the idea of Sadosky and Wheeden in \cite{SaWh}, we prove in Theorem~\ref{generalvanishingmoments} that we can extend the range of $\beta$ in \eqref{generalweightedconds} for which inequality \eqref{generalweightednorm} is valid provided that certain moments of $f$ vanish. More precisely, we show that in such case, inequality \eqref{generalweightednorm} holds for some values $\beta>1/q+c_0+c_1$, thus extending the range given in \eqref{generalweightedconds}. Moreover, the assertion is not true in general for $\beta=1/q+c_0+c_1$.  Such statement is a generalization of \cite[Theorem 1]{SaWh}, where, in particular, the authors proved that if a function $f$ has zero mean, then Pitt's inequality \eqref{fouriersharprange} holds (with $n=1$) in the interval $1/q<\beta<1+1/q$ (recall that the optimal range for $\beta$ in the general case is $\max\{0,1/q-1/p'\}\leq \beta<1/q$), and the assertion does not hold for $\beta=1/q$.

 Finally, in Section~\ref{sectionGM} we use the same approach as in Theorem~\ref{suffFtransforms} to derive sufficient conditions for inequality \eqref{generalweightednorm} to hold in the case of functions satisfying general monotonicity conditions. As is known for the case of the sine and Hankel transforms (cf. \cite{DCGT,DLT,LTparis}), the sharp range of $\beta$ for which inequality \eqref{generalweightednorm} holds can be improved in many cases. Our main result of Section~\ref{sectionGM}, Theorem~\ref{thmgeneralGM}, yields the necessary and sufficient conditions for \eqref{generalweightednorm} to hold in the case of the sine and Hankel transforms, as well as for the $\mathscr{H}_\alpha$ transform.

	\section{Preliminary concepts}\label{s2}
	\subsection{The Bessel and Struve functions}\label{ss2.1}
	Here we give useful properties of the Bessel and Struve functions, which can be found in \cite{EMOT,watBesselfn}.
	
We start with the normalized Bessel function. For $x\leq 1$, one has $j_\alpha(x)\asymp 1$, whilst
$$
 j_\alpha(x)=\frac{C_\alpha}{x^{\alpha+1/2}}\cos\bigg(x-\frac{\pi(\alpha+1/2)}{2}\bigg) +O\big(x^{-\alpha-3/2}), \quad  x\to \infty,
 $$
 so that \eqref{besselest} holds. We also need an upper estimate for the primitive function of $x^{\nu}j_\alpha(xy)$, $\nu\in\R$, as a function of $x$. Let $g_{\alpha,y}^\nu$ be such that
 $$
 \frac{d}{dx}g_{\alpha,y}^\nu(x) =x^\nu j_\alpha(xy),
 $$
with constant of integration equal to zero. It is proved in \cite[Lemma 2.6]{unifconvHankel} (see also \cite[Lemma 3.1]{DCGT}) that
 \begin{equation}
 \label{eqBesselPrimitiveEstimate}
 |g_{\alpha,y}^\nu (x)|\lesssim \frac{x^{\nu-\alpha-1/2}}{y^{\alpha+3/2}}.
 \end{equation}

 In relation with the $\mathscr{H}_\alpha$ transform, we have the following property concerning the derivatives of $\mathbf{H}_\alpha$:
 \begin{equation}
 \label{derivativesYH}
  \frac{d}{dx}\big(x^{\alpha}\mathbf{H}_\alpha(x)\big)=x^{\alpha}\mathbf{H}_{\alpha-1}(x).
 \end{equation}

Moreover, for $x\leq 1$ and fixed $\alpha$, the estimate $\mathbf{H}_\alpha(x) \asymp x^{\alpha+1}$ holds. Indeed, in view of \eqref{Hseries}, we only need to show that for $x\leq 1$,
$$
\sum_{k=0}^\infty \frac{(-1)^k(x/2)^{2k}}{\Gamma(k+3/2)\Gamma(k+\alpha+3/2)}\asymp 1.
$$
On the one hand, the latter series is absolutely convergent for $x\leq 1$, and thus bounded from above. On the other hand,
$$
\sum_{k=0}^\infty \frac{(-1)^k(x/2)^{2k}}{\Gamma(k+3/2)\Gamma(k+\alpha+3/2)}\geq \frac{1}{\Gamma(3/2)\Gamma(\alpha+3/2)}\bigg(1-\frac{x^2}{10(\alpha+5/2)}\bigg) \asymp 1, \qquad x\leq 1.
$$

 For large $x$, we have the following asymptotic expansion \cite[p. 332]{watBesselfn}:
 \begin{align}
 \mathbf{H}_\alpha(x)&= 	\bigg(\frac{\pi x}{2}\bigg)^{-1/2}(\sin(x-\alpha\pi/2-\pi/4))+\frac{(x/2)^{\alpha-1}}{\Gamma(\alpha+1/2)\Gamma(1/2)}(1+O(x^{-2})), \label{eqEstimateStruveLarge}
 \end{align}
 from which we can deduce
 \begin{align*}
 |\mathbf{H}_\alpha(x)|&\lesssim x^{-1/2}+x^{\alpha-1} \asymp x^{\max\{-1/2,\alpha-1\}} =\begin{cases} x^{-1/2}, & \text{if }\alpha<1/2, \\
 x^{\alpha-1}, &\text{if }\alpha\geq 1/2,\end{cases}%\label{Hlargevar}
 \end{align*}
so that \eqref{Hest} holds.

\begin{remark}
	\label{remarkHbigalpha}
	It is worth mentioning that for $\alpha\geq 1/2$ and $x>0$,  $\mathbf{H}_\alpha(x)$ is nonnegative \cite[p. 337]{watBesselfn}, and moreover, it easily follows from  \eqref{eqEstimateStruveLarge} that if $\alpha>1/2$, then there is $x_0>1$ such that
		\begin{equation*}
		\label{eqEstimateHspecial}
		\mathbf{H}_\alpha(x) \asymp x^{\alpha-1}, \qquad x>x_0.
		\end{equation*}
		Hence, for such choice of $\alpha$ one has $\mathbf{H}_\alpha(x)\asymp \min\{x^{\alpha+1},x^{\alpha-1}\}$.
\end{remark}

\subsection{Auxiliary lemma}

In the same spirit as in \eqref{eqBesselPrimitiveEstimate}, we  need an upper estimate for the primitive function of $x^\nu \mathbf{H}_\alpha(xy)$. Let us denote 
$$h^\nu_{\alpha,y}(x)=\int_0^x t^\nu \mathbf{H}_\alpha(ty)\, dt, \qquad \nu \geq 1/2, \quad \alpha>-1/2.$$
Then it follows from \eqref{Hest} and the fundamental theorem of calculus that
 $$
 \frac{d}{dx}h^\nu_{\alpha,y}(x)=x^\nu \mathbf{H}_\alpha(xy).
 $$			
 \begin{lemma}\label{LemmaprimitiveStruve}We have, for any $\nu \geq 1/2$ and $\alpha>-1/2$,
 	$$ 	|h_{\alpha,y}^\nu(x)|\lesssim y^{-1}x^\nu \min\{(xy)^{\alpha+2},(xy)^\alpha\}.$$
 \end{lemma}
 \begin{proof}By definition of $h_{\alpha,y}^\nu$,
 	$$
 h_{\alpha,y}^\nu(x)=\int_0^x t^\nu \mathbf{H}_\alpha(ty)\, dt=\frac{1}{y}\int_0^{xy} \bigg(\frac{z}{y}\bigg)^\nu \mathbf{H}_\alpha(z)\, dz=\frac{1}{y^{\nu+1}} \int_{0}^{xy} z^{\nu-\alpha-1} z^{\alpha+1}\mathbf{H}_\alpha(z)\, dz,
 	$$
 	where we have applied the change of variable $z=ty$. 
 	
 	If $\nu=\alpha+1$, then we simply have $h_{\alpha,y}^{\alpha+1}(x)=y^{-1} x^{\alpha+1}\mathbf{H}_{\alpha+1}(xy)$, by \eqref{derivativesYH}, and \eqref{Hest} implies that
 	$$
 	|h_{\alpha,y}^{\alpha+1}(x)|\lesssim y^{-1}x^{\alpha+1}\min\{(xy)^{\alpha+2},(xy)^{\alpha}\}
 	$$
 	If $\nu\neq \alpha+1$, integration by parts along with \eqref{derivativesYH} yields
 	$$
 	h_{\alpha,y}^\nu(x) = \frac{1}{y}x^{\nu}\mathbf{H}_{\alpha+1}(xy)-\frac{(\nu-\alpha-1)}{y^{\nu+1}}\int_0^{xy}z^{\nu-1}\mathbf{H}_{\alpha+1}(z)\, dz=:A-B.
 	$$
 	Let us now estimate $A$ and $B$ (recall that since $\alpha>-1/2$, $\mathbf{H}_{\alpha+1}$ is nonnegative). On the one hand,
 	$$
 	A\lesssim \frac{1}{y}x^\nu \min\{(xy)^{\alpha+2},(xy)^{\alpha}\}.
 	$$
 	On the other hand, we consider two cases in order to estimate $B$. If $xy\leq 1$, we have
 	$$
 	|B|\lesssim \frac{1}{y^{\nu+1}} \int_0^{xy} z^{\nu+\alpha+1}\, dz \asymp \frac{(xy)^{\nu+\alpha+2}}{y^{\nu+1}}= \frac{1}{y}x^{\nu} (xy)^{\alpha+2}.
 	$$
 	If $xy\geq 1$,
 	$$
 	|B|\lesssim \frac{1}{y^{\nu+1}}\int_0^{xy} z^{\nu+\alpha-1}\, dz \asymp \frac{1}{y}x^\nu (xy)^\alpha.
 	$$
 	Collecting these estimates, we conclude
 	$$
 	|h_{\alpha,y}^\nu(x)|\lesssim A+|B|\lesssim \frac{1}{y}x^\nu \min\{(xy)^{\alpha+2},(xy)^\alpha\},
 	$$
 	as desired.
 \end{proof}

	\section{Weighted norm inequalities for integral transforms}\label{S3}
	
	In this section, we aim to study what necessary and sufficient conditions should nonnegative weights $u,v$ satisfy for the weighted norm inequality \eqref{eqpittDLT} to hold for the transform $F$ given by \eqref{transform}, assuming only estimate \eqref{kernelest}.

	\subsection{Sufficiency results}

	In order to prove Theorem~\ref{suffFtransforms}, we make use of Hardy's inequality (cf. \cite{hardyineq}). If $p=1$, $q=\infty$ or $p=q=\infty$, the result holds under the usual modification of $L^p$ norms.				
	\begin{lemma}\label{hardylema}
		
		Let $1\leq  p \leq q \leq \infty$. If $u,v$ are nonnegative, there exists $B>0$ such that the inequality
		$$
		\bigg( \int_0^\infty u(y) \bigg(\int_0^y |g(x)|\, dx\bigg)^q dy\bigg)^{1/q}\leq B\bigg(\int_0^\infty v(x)|g(x)|^p\, dx \bigg)^{1/p}
		$$
		holds for every measurable $g$ if and only if there exists $C>0$ such that for every $r>0$,
		$$
		\bigg( \int_r^\infty u(y) \,dy\bigg)^{1/q}\bigg(\int_0^r v(x)^{1-p'}\, dx\bigg)^{1/p'}\leq C.
		$$
		Also, there exists $B>0$ such that the inequality
		$$
		\bigg(\int_0^\infty u(y) \bigg(\int_y^\infty |g(x)|\, dx\bigg)^q dy\bigg)^{1/q}\leq B\bigg(\int_0^\infty v(x)|g(x)|^p\, dx\bigg)^{1/p}
		$$
		holds for every measurable $g$ if and only if there exists $C>0$ such that for every $r>0$,
		$$
		\bigg( \int_0^r u(y)\, dy\bigg)^{1/q} \bigg(\int_r^\infty v(x)^{1-p'}\, dx\bigg)^{1/p'}\leq C.
		$$
	\end{lemma}			
	
	\begin{proof}[Proof of Theorem~\ref{suffFtransforms}]
		It follows from \eqref{pointwiseest} and the change of variables $y\to 1/y$ that
		\begin{align*}
		\Vert w^{1/a'} Ff\Vert_{q,u} &\lesssim 
		  \bigg(\int_0^\infty u(1/y) w(1/y)^{q/a'} y^{-2}\bigg(\int_0^y s(x)|f(x)|\, dx \bigg)^q dy\bigg)^{1/q}\\
		&\phantom{=}+\bigg(\int_{0}^\infty u(1/y) w(1/y)^{q(1/a'-1/2)} y^{-2}\bigg(\int_y^\infty s(x)^{1/2}|f(x)|\, dx\bigg)^q dy\bigg)^{1/q}=:I_1+I_2.
		\end{align*}
		We proceed to estimate $I_1$ and $I_2$ from above. Applying Lemma~\ref{hardylema} with $g(x)=s(x) f(x)$, we have the estimate
		$$
		I_1=\bigg(\int_0^\infty u(1/y) w(1/y)^{q/a'}y^{-2}\bigg(\int_0^y s(x)|f(x)|\, dx\bigg)^q dy\bigg)^{1/q} \lesssim \bigg( \int_0^\infty v(x) s(x)^{p/a}|f(x)|^p\, dx \bigg)^{1/p},
		$$
		provided that
		$$
		\bigg(\int_{r}^\infty u(1/y)w(1/y)^{q/a'}y^{-2}\, dy\bigg)^{1/q}\bigg( \int_0^r v(x)^{1-p'} s(x)^{p'/a'}\,dx\bigg)^{1/p'}\leq C,\qquad r>0
		$$
		is satisfied, or equivalently, if \eqref{sufficiencycond1} holds. Finally, if \eqref{sufficiencycond2} holds, or equivalently, if
		$$
		\bigg(\int_0^r u(1/y)w(1/y)^{q(1/a'-1/2)}y^{-2}\, dy\bigg)^{1/q}\bigg(\int_r^\infty v(x)^{1-p'} s(x)^{p'(1/a'-1/2)}\, dx \bigg)^{1/p'}\leq C, \qquad r>0,
		$$
		then, applying Lemma~\ref{hardylema} with $g(x)=s(x)^{1/2}f(x)$, we obtain the estimate
		\begin{align*}
		I_2&=\bigg(\int_0^\infty u(1/y)w(1/y)^{q(1/a'-1/2)}y^{-2}\bigg(\int_y^\infty s(x)^{1/2}|f(x)|\, dx\bigg)^q dy\bigg)^{1/q} \\
		&\lesssim \bigg(\int_0^\infty v(x) s(x)^{p/a}|f(x)|^p\, dx\bigg)^{1/p},
		\end{align*}
		which establishes inequality \eqref{eqpittDLT}.
	\end{proof}
	In order to prove Corollary~\ref{CORgluing} we first show a generalization of the gluing lemma \cite[Lemma 2.2]{GKP}:
	\begin{lemma}\label{lemmaglue}
		Let $f,g\geq 0$, $\alpha,\beta>0$ and let $\varphi,\psi$ be nonnegative and nonincreasing. Assume $\varphi(s)^\alpha \asymp \psi(s)^\beta$. Then, the conditions 
		\begin{equation}\label{gluing1}
		\sup_{t>0} \bigg(\int_0^t g(s)\, ds\bigg)^\beta  \bigg( \int_t^\infty \varphi(s)f(s)\, ds\bigg)^\alpha<\infty
		\end{equation}
		and
		\begin{equation}
		\label{gluing2}
		\sup_{t>0}\bigg( \int_0^t f(s)\, ds\bigg)^\alpha \bigg(\int_t^\infty \psi(s)g(s)\, ds\bigg)^\beta<\infty
		\end{equation}
		hold simultaneously if and only if 
		\begin{equation}
		\label{gluing3}
		\sup_{t>0}
		\bigg(\int_0^t g(s)\, ds+\frac{1}{\psi(t)}\int_t^\infty \psi(s)g(s)\, ds \bigg)^\beta\bigg( \varphi(t)\int_0^t
		f(s)\, ds + \int_t^\infty \varphi(s) f(s)\, ds\bigg)^\alpha  <\infty.	\end{equation}
	\end{lemma}
	\begin{proof}[Proof of Lemma~\ref{lemmaglue}]
		It is clear that \eqref{gluing3} is equivalent to the finiteness of
		\begin{align}
		\sup_{t>0}\bigg[& \varphi(t)^\alpha \bigg(\int_0^t g(s)\, ds \bigg)^\beta \bigg(\int_0^t f(s)\, ds\bigg)^\alpha + \bigg(\int_0^t g(s)\, ds\bigg)^\beta\bigg(\int_t^\infty \varphi(s) f(s)\, ds\bigg)^\alpha \nonumber \\&  +\frac{\varphi(t)^\alpha}{\psi(t)^\beta} \bigg(\int_t^\infty \psi(s)g(s) \,ds \bigg)^\beta \bigg(\int_0^t f(s)\, ds\bigg)^\alpha \nonumber \\ & +\frac{1}{\psi(t)^\beta} \bigg(\int_t^\infty \psi(s) g(s)\, ds \bigg)^\beta \bigg(\int_t^\infty \varphi(s) f(s)\, ds \bigg)^\alpha\bigg].\label{gluingsup}
		\end{align}
		From the latter it is obvious that \eqref{gluing3} implies \eqref{gluing1} and \eqref{gluing2}, since $\varphi(t)^\alpha/\psi(t)^\beta\asymp 1$. In order to prove the converse, note that the second term of \eqref{gluingsup} corresponds to \eqref{gluing1}, whilst the third term of \eqref{gluingsup} corresponds to \eqref{gluing2} (after applying the equivalence $\varphi(t)^\alpha \asymp \psi(t)^\beta$ on the term outside the integrals). Thus, it remains to prove the finiteness of the first and fourth terms of \eqref{gluingsup}. For $t>0$, let $b(t)\in (0,t)$ be the number such that $\int_0^{b(t)} f(s)\, ds = \int_{b(t)}^t f(s)\, ds$. Then, using the monotonicity of $\varphi$ and $\psi$, and the equivalence $\varphi(s)^\alpha\asymp \psi(s)^\beta$, we get
		\begin{align*}
		&\phantom{\asymp}\varphi(t)^\alpha\bigg(\int_0^t g(s)\, ds\bigg)^\beta \bigg(\int_0^t f(s)\, ds\bigg)^\alpha \\
		&\asymp \varphi(t)^\alpha \bigg(\int_0^{b(t)} g(s)\, ds\bigg)^\beta \bigg(\int_0^t f(s)\, ds\bigg)^\alpha 
		+ \psi(t)^\beta \bigg( \int_{b(t)}^t g(s)\, ds\bigg)^\beta \bigg(\int_0^t f(s)\, ds\bigg)^\alpha \\
		&\leq \bigg(\int_0^{b(t)} g(s)\, ds\bigg)^\beta \bigg(\int_{b(t)}^t \varphi(s)f(s)\, ds \bigg)^\alpha 	+ \bigg(\int_{b(t)}^t \psi(s) g(s)\, ds\bigg)^\beta \bigg(\int_0^{b(t)} f(s)\, ds\bigg)^\alpha \\ 
		&\leq \sup_{t>0} \bigg(\int_0^t g(s)\, ds \bigg)^\beta \bigg(\int_t^\infty \varphi(s) f(s) \, ds\bigg)^\alpha + \sup_{t>0}\bigg(\int_t^\infty \psi(s) g(s)\, ds \bigg)^\beta \bigg(\int_0^t f(s)\, ds\bigg)^\alpha <\infty.
		\end{align*}
		Similarly, for $t\in (0,\infty)$, let $c(t)\in (t,\infty)$ be such that $\int_t^{c(t)} \psi(s)g(s)\, ds =\int_{c(t)}^\infty \psi(s)g(s)\, ds$. We have
		{\small
		\begin{align*}
		&\phantom{\asymp} \frac{1}{\psi(t)^\beta}\bigg( \int_t^\infty \psi(s)g(s)\, ds\bigg)^\beta \bigg(\int_t^\infty \varphi(s)f(s)\, ds\bigg)^\alpha \\
		& \asymp \frac{1}{\psi(t)^\beta}\bigg( \bigg( \int_t^\infty \psi(s)g(s)\, ds\bigg)^\beta \bigg(\int_t^{c(t)} \varphi(s)f(s)\, ds\bigg)^\alpha + \bigg( \int_t^\infty \psi(s)g(s)\, ds\bigg)^\beta \bigg(\int_{c(t)}^\infty \varphi(s)f(s)\, ds\bigg)^\alpha\bigg)\\
		&\asymp \frac{1}{\psi(t)^\beta} \bigg( \bigg(\int_{c(t)}^\infty \psi(s)g(s)\, ds\bigg)^\beta \bigg( \int_t^{c(t)} \varphi(s)f(s)\, ds  \bigg)^\alpha +  \bigg(\int_t^{c(t)} \psi(s)g(s)\, ds\bigg)^\beta \bigg(\int_{c(t)}^\infty \varphi(s)f(s)\, ds\bigg)^\alpha \bigg)\\
		&\leq \bigg(\int_{c(t)}^\infty \psi(s)g(s)\, ds \bigg)^\beta \bigg(\int_t^{c(t)} f(s)\, ds\bigg)^\alpha + \bigg(\int_t^{c(t)} g(s)\, ds\bigg)^\beta \bigg(\int_{c(t)}^\infty \varphi(s)f(s)\, ds \bigg)^\alpha\\
		&\leq \sup_{t>0} \bigg(\int_t^\infty \psi(s)g(s)\, ds\bigg)^\beta \bigg(\int_0^t f(s)\, ds\bigg)^\alpha+\sup_{t>0} \bigg(\int_0^t g(s)\, ds\bigg)^\beta \bigg(\int_t^\infty \varphi(s)f(s)\, ds\bigg)^\alpha<\infty,
		\end{align*}
	}
		as desired.
	\end{proof}
	\begin{proof}[Proof of Corollary~\ref{CORgluing}]
		Note that we can rewrite conditions \eqref{sufficiencycond1} and \eqref{sufficiencycond2} (with $a=1$) as
		\begin{align}
			\sup_{r>0} \bigg( \int_r^\infty u(1/y)y^{-2}\, dy\bigg)^{1/q}\bigg(\int_0^r v(y)^{1-p'}\, dy\bigg)^{1/p'}&<\infty, \label{EQgluin-aux1}
			\\
			\sup_{r>0} \bigg( \int_0^r u(1/y)w(1/y)^{-q/2} y^{-2}\, dy\bigg)^{1/q}\bigg(\int_r^\infty v(y)^{1-p'}s(y)^{-p'/2}\, dy\bigg)^{1/p'}&<\infty.\label{EQgluin-aux2}
		\end{align}
		Putting
		$$
		f(y)=u(1/y)w(1/y)^{-q/2}y^{-2}, \qquad g(y)=v(y)^{1-p'},
		$$
		together with $\varphi(y)=w(1/y)^{q/2}$, $\psi(y)=s(y)^{-p'/2}$,  $\alpha=1/q$ and $\beta=1/p'$, it is clear that \eqref{EQgluin-aux1} and \eqref{EQgluin-aux2} are the same as \eqref{gluing1} and \eqref{gluing2} respectively. Also, observe that \eqref{inverse} is equivalent to $\varphi(y)^\alpha \asymp \psi(y)^\beta$. Hence, we are under the hypotheses of Lemma~\ref{CORgluing}, and we can deduce that the joint fulfilment of \eqref{sufficiencycond1} and \eqref{sufficiencycond2} is equivalent to 
		\begin{align*}
		\sup_{t>0}\bigg[ \bigg( \int_0^t v(y)^{1-p'}\, dy +& s(t)^{p'/2}\int_t^\infty v(y)^{1-p'}s(y)^{-p'/2}\, dy\bigg)^{1/p'}\\
		&\times \bigg( w(1/t)^{q/2}\int_0^t u(1/y)w(1/y)^{-q/2}y^{-2}\, dy + \int_t^\infty u(1/y)y^{-2}\, dy  \bigg)^{1/q}\bigg]<\infty,
		\end{align*}
		or equivalently, \eqref{EQglued}.
	\end{proof}
	
	\subsection{Necessity results}
	\subsubsection{Necessity in weighted Lebesgue spaces}
	Here we present necessary conditions for \eqref{eqpittDLT} to hold, with $F$ given by \eqref{transform}. We consider the following assumptions on the weights $u,v$:
	$$
	u w^{q/a'}\in L^1_{\textrm{loc}}, \qquad v^{1-p'} s^{p'/a'}\in L^1_{\textrm{loc}}.
	$$
	\begin{theorem}\label{thmnecessity1}
		Let $1<p,q<\infty$ and $1\leq a\leq \infty$. Assume that  inequality \eqref{eqpittDLT} holds for every $f$, where
		$$
		Ff(y)=\int_0^\infty s(x) f(x)K(x,y)\, dx.
		$$
		\begin{enumerate}[label=(\roman{*})]
			\item If the kernel $K(x,y)$ satisfies
		\begin{equation}
		\label{kernel=1}
		K(x,y)\asymp 1, \qquad 0< xy\leq 1,
		\end{equation}
		then \eqref{sufficiencycond1} is valid.
		\item If the kernel $K(x,y)$ satisfies
				\begin{equation*}
				\label{kernel=atinfinity}
				K(x,y)\asymp (s(x)w(y))^{-1/2}, \qquad xy> 1,
				\end{equation*}
				then \eqref{sufficiencycond2} is valid.
	\end{enumerate}
	\end{theorem}
	Note that Corollary~\ref{coriff} readily follows from Theorems~\ref{suffFtransforms}~and~\ref{thmnecessity1}.

	\begin{proof}[Proof of Theorem~\ref{thmnecessity1}]
		For the first part, let 
		$$
		f_r(x)=v(x)^{1-p'}s(x)^{(1-p')(p/a-1)}\chi_{(0,r)}(x), \qquad r>0.
		$$
		It follows from \eqref{kernel=1} and the equality $1+(1-p')(p/a-1)=p'/a'$ that for $y\leq 1/r$
		$$
		|Ff_r(y)| =\int_0^\infty f_r(x) K(x,y)s(x)\, dx \asymp \int_0^r v(x)^{1-p'}s(x)^{p'/a'}\, dx.		$$
		On the one hand, we have
		\begin{equation}
		\label{lowernormest}
		\big\Vert w^{1/a'} Ff\big\Vert_{q,u}\geq \bigg(\int_{0}^{1/r} u(y) w(y)^{q/a'}\, dy\bigg)^{1/q}\bigg(\int_0^r v(x)^{1-p'}s(x)^{p'/a'}\, dx\bigg),
		\end{equation}
		and on the other hand, 
		$$
		\big\Vert s^{1/a} f_r\big\Vert_{p,v}=\bigg(\int_0^r v(x)^{1-p'} s(x)^{p'/a'}\, dx\bigg)^{1/p}.
		$$
		Combining the latter equality with \eqref{eqpittDLT} and \eqref{lowernormest}, we derive
		$$
		\bigg(\int_{0}^{1/r} u(y) w(y)^{q/a'}\, dy\bigg)^{1/q}\bigg(\int_0^r v(x)^{1-p'}s(x)^{p'/a'}\, dx\bigg) \lesssim \bigg(\int_0^r v(x)^{1-p'}s(x)^{p'/a'}\, dx\bigg)^{1/p},
		$$
		i.e., \eqref{sufficiencycond1} holds.
		
		We omit the proof of the second part, as it is essentially a repetition of that of the first part. In this case, one should consider the function
		$$
		f_r(x)= v(x)^{1-p'}s(x)^{p'(1/a'-1/2)-1/2}\chi_{(r,\infty)}(x), \qquad r>0,
		$$
		and proceed analogously as above.
	\end{proof}
	The latter shows that condition \eqref{sufficiencycond1} is best possible for some classical transforms, such as the Hankel (or the cosine) transform, since $j_\alpha(xy)\asymp 1$ whenever $xy\leq 1$ for every $\alpha\geq -1/2$ (i.e., \eqref{kernel=1} is satisfied). 

\subsubsection{Necessity in weighted Lorentz spaces}
To conclude the part dealing with necessary conditions for \eqref{eqpittDLT}, we present a generalization of a result due to Benedetto and Heinig \cite[Theorem 2]{BH}, related to weighted Lorentz spaces (introduced in \cite{Lorentzspaces}; see also \cite{CRSlorentz}). We also refer the reader to \cite{BozaSoria,NT,sinnamonlorentz} for recent advances in the theory of Fourier inequalities in Lorentz spaces. 

In this part we do not present sufficiency conditions, as those rely on rearrangement inequalities that follow from Bessel's weighted inequality \eqref{besselineq} (cf. \cite{BH}), which we are not considering in this work.

Recall that for a measure space $(X,\mu)$ with $X\subset \R$ and $f$ a complex $\mu$-measurable function, we define the distribution function of $f$ as
$$
D_f(t)=\mu\{x\in X:|f(x)|>t\}, \quad t\in [0,\infty).
$$
Note that $D_f$ is nonnegative. Moreover, for $0<p<\infty$ (see, e.g., \cite{BSbook}),
$$
\int_X |f(x)|^p\, d\mu(x)=p\int_0^\infty t^{p-1}D_f(t)\,dt.
 $$
\begin{theorem}
Let $1<p,q<\infty$. Assume that the kernel $K(x,y)$ from \eqref{transform} satisfies $K(x,y)\asymp 1$ for $xy\leq 1$. If $u$ and $v$ are weights such that the inequality
\begin{equation}\label{benedettoeq1}
\bigg(\int_0^\infty (Ff)^*(y)^qu(y)\, dy\bigg)^{1/q}\leq C_0\bigg(\int_0^\infty f^*(x)^pv(x)\,dx\bigg)^{1/p}
\end{equation}
holds for every $f$, then
$$
 \bigg(\int_0^{1/r} u(y)\, dy\bigg)^{1/q}\bigg(\int_0^{r} v(x)\, dx\bigg)^{-1/p}\bigg( \int_0^r s(x)\, dx\bigg)\leq C, \quad r>0.
$$
\end{theorem}
	\begin{proof}The argument is similar to that of \cite[Theorem 2]{BH}. Let $f(x)=\chi_{(0,r)}(x)$. It is clear that $f^*=f$. Observe that the right hand side of \eqref{benedettoeq1} is equal to $C_0\big(\int_0^r v(x)\, dx\big)^{1/p}$, 	and moreover we have
		$$
		Ff(y)=\int_0^\infty s(x)f(x)K(x,y)\,dx =\int_0^r s(x)K(x,y)\, dx.
		$$
		If we denote $c=\min_{xy\leq 1}K(x,y)$, then by hypotheses $c>0$, and for $y\leq 1/r$, one has
		\begin{equation}\label{estbenedetto}
		Ff(y)>\frac{c}{2}\int_0^{r}s(x)\, dx=: A_r.
		\end{equation}
		For any $r>0$, the following estimate holds:
		\begin{align}
		\bigg( \int_0^\infty (Ff)^*(y)^qu(y)\, dy\bigg)^{1/q} &\geq \bigg( \int_0^{1/r} (Ff)^*(y)^qu(y)\, dy\bigg)^{1/q} \nonumber \\
		&= \bigg(q\int_0^\infty t^{q-1}\bigg(\int_{\{y\in(0,1/r):(Ff)^*(y)>t\}}u(y)\, dy\bigg)dt\bigg)^{1/q}\nonumber \\ 
		&=\bigg(q\int_0^\infty t^{q-1} \bigg(\int_0^{\min\{D_{Ff}(t),1/r\}} u(y)\, dy\bigg)dt\bigg)^{1/q}.\label{benedettoeq2}
		\end{align}
		where in the last step we have used that $\{y:(Ff)^* (y)>t\}=\{y:D_{Ff}(t)>y\}$. Also note that for $t<A_r$, \eqref{estbenedetto} implies
		$$
		(0,1/r)\subset \{ y>0 : Ff(y)>A_r\}\subset \{y>0: Ff(y)>t\}.
		$$
		Thus, for such choice of $t$,
		$$
		D_{Ff}(t)=\int_{\{y>0: |Ff(y)|>t\}}dz\geq \int_0^{1/r}dy =\frac{1}{r}.
		$$
		In view of the latter, we deduce that if $t<A_r$, then $\min\{D_{Ff}(t),1/r\}=1/r$. Combining such observation with \eqref{benedettoeq2}, we obtain
		\begin{align*}
		\bigg(\int_0^\infty (Ff)^*(y)^qu(y)\, dy\bigg)^{1/q}&\geq \bigg(q \int_0^{A_r} t^{q-1}\bigg(\int_0^{1/r}u(y)\, dy\bigg)dt\bigg)^{1/q} =A_r \bigg(\int_0^{1/r}u(y)\, dy\bigg)^{1/q}.
		\end{align*}
		Finally, it follows from \eqref{benedettoeq1} and the previous estimates that
		\begin{align*}
		\frac{c}{2} \bigg(\int_0^{1/r} u(y)\, dy\bigg)^{1/q}&\bigg(\int_0^{r} v(x)\, dx\bigg)^{-1/p}\bigg( \int_0^r s(x)\, dx\bigg)\\
		&\leq \bigg(\int_0^\infty (Ff)^*(y)^q u(y)\, dy\bigg)\bigg(\int_0^r v(x)\, dx\bigg)^{-1/p}\\
		&\leq C_0\bigg(\int_{0}^{r} v(x)\, dx\bigg)^{1/p}\bigg(\int_0^r v(x)\, dx\bigg)^{-1/p}= C_0,
		\end{align*}which establishes the assertion.
		\end{proof}

	\section{Weighted norm inequalities for transforms with power-type kernel}\label{S4}
	
	In what follows we assume $u(x)=x^{-\overline{\beta}'q}$, $v(x)=x^{\overline{\gamma}p}$ with $\beta_1-\gamma_1=\beta_2-\gamma_2$ and $s(x)=w(x)=x^\delta$, $\delta>0$ in \eqref{transform}. Piecewise power weights have been considered for the study of weighted restriction Fourier inequalities \cite{BS,DCGTrestr}, and moreover they play a fundamental role in the study of weighted norm inequalities for the Jacobi transform \cite{DLT} (see also \cite{jacobi}).
	
	For the sake of generality, we first give sufficient conditions for \eqref{pittpower} to hold, and then we also study necessary  conditions for \eqref{generalweightednorm} to hold, i.e., with non-mixed power weights.
	
	\subsection{Sufficient conditions}
		\begin{proof}[Proof of Theorem~\ref{theoremsuffpowerweights}]
		Let us verify that conditions \eqref{sufpowers} and \eqref{sufpowers2} imply \eqref{sufficiencycond1} and \eqref{sufficiencycond2} with $a=1$. On the one hand, it is clear that the integrals on the left hand side of \eqref{sufficiencycond1} converge if and only if
		$$
		\beta_2<\frac{1}{q} \qquad \text{and}\qquad \gamma_1<\frac{1}{p'}.
		$$
		On the other hand, the integrals on the left hand side \eqref{sufficiencycond2} converge if and only if
		$$
				\beta_1>\frac{1}{q}-\frac{\delta}{2} \qquad\text{and}\qquad  \gamma_2>\frac{1}{p'}-\frac{\delta}{2}.
		$$
		Notice that \eqref{sufpowers} and \eqref{sufpowers2} (along with $\beta_1-\gamma_1=\beta_2-\gamma_2$) imply that all the previous conditions hold. Now we proceed to verify that \eqref{sufficiencycond1} and \eqref{sufficiencycond2} hold. It suffices to check those conditions for $r< 1/2$ or $r> 2$. We check \eqref{sufficiencycond1} first. If $r<1/2$,
		\begin{align*}
		\bigg( \int_0^{1/r} u(y) \, dy\bigg)^{1/q}\bigg(\int_0^r v(x)^{1-p'} \, dx\bigg) ^{1/p'} & \asymp r^{-\gamma_1+1/p'} \bigg(C+\int_1^{1/r} y^{-\beta_1 q}\, dy\bigg)^{1/q}  \\ & \asymp r^{-\gamma_1+1/p'}\max\{ 1, r^{\beta_1-1/q}\} \\ &= \max\{ r^{-\gamma_1+1/p'},r^{\beta_1-\gamma_1+1/p'-1/q} \},
		\end{align*}
		which is uniformly bounded in $r<1/2$ if and only if
		\begin{equation}\label{mixedcond1}
		\beta_1-\gamma_1\geq 1/q-1/p'.
		\end{equation}
		If $r> 2$, 
		\begin{align*}
		\bigg( \int_0^{1/r} u(y) \, dy\bigg)^{1/q}\bigg(\int_0^r v(x)^{1-p'} \, dx\bigg) ^{1/p'} &\asymp r^{\beta_2-1/q}\bigg( C+\int_1^{r} x^{\gamma_2 p(1-p')}\, dx\bigg)^{1/p'} \\
		& \asymp r^{\beta_2-1/q} \max\{ 1, r^{-\gamma_2+1/p'}\} \\ &= \max\{ r^{\beta_2-1/q}, r^{ \beta_2-\gamma_2+1/p'-1/q}\}.
		\end{align*}
		The latter is uniformly bounded in $r>2$ if and only if
		\begin{equation}
		\label{mixedcond2}
		\beta_2-\gamma_2\leq 1/q-1/p'.
		\end{equation}
		The joint fulfilment conditions \eqref{mixedcond1} and \eqref{mixedcond2} together with $\beta_1-\gamma_1=\beta_2-\gamma_2$ is equivalent to \eqref{sufpowers}.
		
		Finally, we are left to verify \eqref{sufficiencycond2}. First, if $r<1/2$,
			\begin{align*}
			\bigg(\int_{1/r}^\infty u(y) y^{-q\delta/2}\, dy\bigg)^{1/q}\bigg( \int_r^\infty v(x)^{1-p'} x^{-p'\delta/2}\, dx\bigg)^{1/p'}&\asymp r^{\beta_1+\delta/2-1/q}\bigg(C+\int_r^1 x^{-p'(\gamma_1+ \delta/2)}\, dx\bigg)^{1/p'} \\ 
			&\asymp\max\{ x^{\beta_1+\delta/2-1/q},r^{\beta_1-\gamma_1+1/p'-1/q}\},					
			\end{align*}		
			which is uniformly bounded in $r<1/2$ if and only if \eqref{mixedcond1} holds. Secondly, for $r> 2$,
			\begin{align*}
			\bigg(\int_{1/r}^\infty u(y) y^{-q\delta/2}\, dy\bigg)^{1/q}\bigg( \int_r^\infty v(x)^{1-p'} x^{-p'\delta/2}\, dx\bigg)^{1/p'}&\asymp r^{-\gamma_2-\delta/2+1/p'}\bigg(C+\int_{1/r}^1 y^{-q(\beta_2+\delta/2)}\, dy\bigg)^{1/q} \\
			& \asymp \max\{ r^{-\gamma_2-\delta/2+1/p'},r^{\beta_2-\gamma_2+1/p'-1/q}\},
			\end{align*}
			and the latter is uniformly bounded in $r>2$ if and only if \eqref{mixedcond2} holds.
	\end{proof}
	
	Now we prove Corollary~\ref{corapplications}, which applies to  transforms with power-type kernel, and is equivalent to Theorem~\ref{theoremsuffpowerweights} with non-mixed power weights, as mentioned in the Introduction.

	\begin{proof}[Proof of Corollary~\ref{corapplications}]
		The proof is essentially based on changing variables in Theorem~\ref{theoremsuffpowerweights}. Let us define $d=c_1-c_2$ and
		$$
		\widetilde{K}(x,y)=x^{-b_1}y^{-c_1}K(x,y).
		$$
		Then it holds that
		$$
		|\widetilde{K}(x,y)|\lesssim \min\big\{1,(xy)^{-d}\big\}, \qquad d>0.
		$$
		We also define the auxiliary integral transform
		$$
		Gf(y)= \int_0^\infty x^{2d}f(x)\widetilde{K}(x,y)\, dx,
		$$
		which satisfies the hypotheses of Theorem~\ref{theoremsuffpowerweights} (with $\delta=2d$). Putting $g(x)=x^{b_0+b_1-2d}f(x)$, we have the following relation:
		$$
		y^{c_0+c_1}Fg(y)=Gf(y),
		$$
		and therefore, in virtue of Theorem~\ref{theoremsuffpowerweights}, the weighted norm inequality
		$$
		\Vert y^{-c_0-c_1-\beta'}Gf\Vert_q =\Vert y^{-\beta'} Fg \Vert_q \lesssim \Vert x^{\gamma'+2d} g\Vert_p = \Vert x^{\gamma' +b_0+b_1}f\Vert_p, \qquad 1<p\leq q<\infty,
		$$
		holds with $\beta'=\gamma'+1/q-1/p'$ and $1/q-d<\beta'<1/q$, or in other words, if we set $\beta=\beta'+c_0+c_1$ and $\gamma=\gamma'+b_0+b_1$, then inequality \eqref{generalweightednorm} holds if both conditions in  \eqref{generalweightedconds} are satisfied.
	\end{proof}

At this point we can already derive sufficient conditions for \eqref{generalweightednorm} to hold whenever $F$ is the sine, Hankel, or $\mathscr{H}_\alpha$ transform. To do this, we use the estimates \eqref{besselest}, \eqref{Hest}, and Corollary~\ref{corapplications} (recall that $|\sin xy|\leq \min\{xy,1\}$ for $x,y>0$). Those sufficient conditions are the following:
\begin{itemize}
	\item Sine transform: $\beta=\gamma+1/q-1/p'$, and
	$$
	\frac{1}{q}<\beta<1+\frac{1}{q}.
	$$
	\item Hankel transform of order $\alpha>-1/2$: $\beta=\gamma-2\alpha-1+1/q-1/p'$, and
	$$
	\frac{1}{q}-\alpha-\frac{1}{2}<\beta<\frac{1}{q}.
	$$
	\item $\mathscr{H}_\alpha$ transform of order $\alpha>-1/2$: $\beta=\gamma+1/q-1/p'$, and
	\begin{align*}
	\frac{1}{q}<\beta&<\frac{1}{q}+\alpha+\frac{3}{2}, &\text{if }\alpha<1/2,\\
	\frac{1}{q}+\alpha-\frac{1}{2}<\beta&<\frac{1}{q}+\alpha+\frac{3}{2}, &\text{if }\alpha\geq 1/2.
	\end{align*}
\end{itemize}
Note that the above conditions are not optimal in the case of the sine and Hankel transforms. For the sine transform, it is known \cite{DLT} that \eqref{generalweightednorm} holds if and only if $\beta=\gamma+1/q-1/p'$ and
$$
\max\bigg\{\frac{1}{q}-\frac{1}{p'},0\bigg\}\leq \beta<1+\frac{1}{q},
$$
and for the Hankel transform (of order $\alpha\geq -1/2$), \eqref{generalweightednorm} holds if and only if (see \cite{DC}) $\beta=\gamma-2\alpha-1+1/q-1/p'$ and
$$
\max\bigg\{ \frac{1}{q}-\frac{1}{p'},0 \bigg\} -\alpha-\frac{1}{2}\leq \beta<\frac{1}{q}.
$$

For the $\mathscr{H}_\alpha$ transform, Rooney proved \cite{RonCan} that \eqref{generalweightednorm} holds if $\beta=\gamma+1/q-1/p'$ and
\begin{align}
 \beta\geq \max\bigg\{ \frac{1}{q}-\frac{1}{p'},0\bigg\} \qquad \text{and}\qquad  \frac{1}{q}+\alpha-\frac{1}{2}<\beta&<\frac{1}{q}+\alpha+\frac{3}{2}, &\text{if }\alpha&<1/2, \nonumber \\
\frac{1}{q}+\alpha-\frac{1}{2}<\beta&<\frac{1}{q}+\alpha+\frac{3}{2}, &\text{if }\alpha&\geq 1/2.\label{suffHtrans}
\end{align}
Note that whenever $\alpha> 1/2$, the above sufficient conditions coincide with those given by Corollary~\ref{corapplications}, and moreover they are optimal, (see Theorem~\ref{ThmNecessitypowers} and Remark~\ref{rmksharpness} below).

	\subsection{Necessary conditions}
	Let us now study what conditions follow from \eqref{generalweightednorm}. The main result of this subsection goes along the same lines as Theorem~\ref{thmnecessity1}.
		\begin{theorem}
			\label{ThmNecessitypowers} Let $1<p\leq q<\infty$. Assume that inequality \eqref{generalweightednorm} holds for all $f$, with $F$ as in \eqref{powertrans}.
			\begin{enumerate}[label=(\roman{*})]
				\item If the kernel $K(x,y)$ satisfies
				$$
				K(x,y)\asymp x^{b_1}y^{c_1}, \qquad xy\leq 1, \quad b_1,c_1\in\R,
				$$
				then
				$$
				\beta=\gamma+c_0-b_0+c_1-b_1+\frac{1}{q}-\frac{1}{p'}, \qquad \beta<\frac{1}{q}+c_0+c_1.
				$$
				\item If the kernel $K(x,y)$ satisfies 
				$$
				K(x,y)\asymp x^{b_2}y^{c_2}, \qquad xy > 1, \quad b_2,c_2\in \R,
				$$
				then
				$$
				\beta=\gamma+c_0-b_0+c_2-b_2+\frac{1}{q}-\frac{1}{p'}, \qquad \beta >\frac{1}{q}+c_0+c_2.
				$$
			\end{enumerate}
		\end{theorem}
		\begin{proof}
		For $r>0$, let $f_r(x)=x^{-b_0-b_1+d}\chi_{(0,r)}(x)$, where $d>-1$ is such that $\gamma-b_0-b_1+d>-1/p$ for a given $\gamma\in \R$. Then
		$$
		\Vert x^{\gamma}f_r\Vert_p =\bigg(\int_0^r x^{p(\gamma-b_0-b_1+d)}\, dx\bigg)^{1/p}\asymp r^{\gamma-b_0-b_1+d+1/p}.
		$$
		If $y\leq 1/r$, one has
		$$
		Ff_r(y)=y^{c_0}\int_0^r x^{-b_1+d}K(x,y)\, dx\asymp r^{d+1}y^{c_0+c_1},
		$$
		Then, it follows from inequality \eqref{generalweightednorm} and the finiteness of $\Vert x^\gamma f_r\Vert_p$ that
		\begin{align*}
		r^{\gamma-b_0-b_1+d+1/p}\asymp \Vert x^{\gamma}f_r\Vert_p  \gtrsim \Vert y^{-\beta}Ff_r\Vert_q & \geq  \bigg( \int_0^{1/r} y^{-\beta q}|Ff_r(y)|^q\, dy\bigg)^{1/q} \\ &\asymp r^{d+1}\bigg(\int_0^{1/r} y^{q(-\beta +c_0+c_1)}\, dy\bigg)^{1/q}\asymp r^{\beta-c_0-c_1-1/q+d+1},
		\end{align*}
		Note that the finiteness of the latter integral is equivalent to $\beta<1/q+c_0+c_1$. Moreover, the inequality $r^{\beta-c_0-c_1-1/q+d+1}\lesssim r^{\gamma-b_0-b_1+d+1/p}$ holds uniformly in $r>0$ if and only if $\beta=c_0-b_0+c_1-b_1+1/q-1/p'$. This completes the proof of the first part.
		
		The proof of the second part is omitted, as it is analogous to that of the first part. In this case one should use consider the function
		$$
		f_r(x)=x^{-b_0-b_2-d}\chi_{(r,\infty)}(x),
		$$
		where $d>1$ is such that $\gamma-b_0-b_2-d<-1/p$ for a given $\gamma\in \R$.
		\end{proof}
		\begin{remark}\label{rmksharpness}
		Note that if the kernel $K(x,y)$ of \eqref{powertrans} is such that
		$$
		K(x,y)\asymp\min\{x^{b_1}y^{c_1},x^{b_2}y^{c_2}\},
		$$
		with $b_1-b_2=c_1-c_2>0$, then the sufficient conditions of Corollary~\ref{corapplications} are also necessary. An example of a transform satisfying such property is the $\mathscr{H}_\alpha$ transform with $\alpha>1/2$ (cf. Remark~\ref{remarkHbigalpha}). This proves that Corollary~\ref{corapplications} is sharp, although in general it does not give the sharp sufficient conditions for inequality \eqref{generalweightednorm} to hold whenever $F$ is some of the aforementioned transforms, such as the sine transform.
		\end{remark}

	\section{Integral transforms with kernel represented by a power series and functions with vanishing moments}\label{SectionVanishingMoments}
	This section is motivated by the well-known result due to Sadosky and Wheeden \cite{SaWh}. They proved that the sufficient conditions \eqref{fouriersharprange} that guarantee Pitt's inequality (in one dimension) can be relaxed, provided that $f$ has vanishing moments. More precisely, one has:
		\begin{theoremA}\label{SaWhthm}
			Let $f$ be such that
			\begin{equation*}
			\label{vanishingmoment}
			\int_{-\infty}^{\infty} x^jf(x)\, dx = 0,\qquad  j=0,\ldots, n-1, \quad n\in \N.
			\end{equation*}
			Then, the weighted norm inequality
			\begin{equation}
			\label{FourierWNI}	
			\bigg(\int_{\R}|x|^{-\beta q}|\widehat{f}(x)|^q\, dx\bigg)^{1/q}\leq C\bigg(\int_{\R} |x|^{\gamma p}|f(x)|^p\,dx\bigg)^{1/p}
		\end{equation}
			holds with $\beta=\gamma +1/q-1/p'$ and
			$$
			\frac{1}{q}<\beta<n+\frac{1}{q}, \qquad \beta \neq \frac{1}{q} + j  , \,j=1,\ldots ,n-1.
			$$
		\end{theoremA}
	\subsection{Main results}
Following the idea of Sadosky and Wheeden, here we obtain an analogous statement to Theorem~\ref{SaWhthm} for transforms with kernels represented by power series. As examples, we mention the sine, Hankel, and $\mathscr{H}_\alpha$ transforms. The generalization of Theorem~\ref{SaWhthm} reads as follows:
\begin{theorem}\label{thmgeneralvanishingmoments}
	Let $1<p\leq q<\infty$ and let the integral transform $F$ be as in \eqref{powertrans}. Let
	\begin{equation}
	\label{kernelseries}
	K(x,y)=x^{b_1}y^{c_1}\sum_{m=0}^\infty a_m (xy)^{km}, \qquad k\in \N, \qquad a_m\in \C, \qquad b_1,c_1\in \R, \qquad x,y>0,
	\end{equation}
	with $\sum_{m=0}^\infty |a_k|=A<\infty$. Assume the series defining $K$ converges for every $x,y>0$, and moreover $|K(x,y)| \lesssim x^{b_2}y^{c_2}$ for $xy>1$, where $b_2,c_2\in \R$, and $c_1-c_2=b_1-b_2\geq 0$. If $f$ is such that
	\begin{equation}
	\label{generalvanishingmoments}
	\int_0^\infty x^{b_0+b_1+\ell k}f(x)\, dx=0, \qquad \ell=0,\ldots , n-1, \qquad n\in \N,
	\end{equation}
	then the inequality $\Vert y^{-\beta}Ff\Vert_q\leq C\Vert x^{\gamma}f\Vert_p$ holds with
	\begin{equation*}
	\label{generalvanishingmomentscond}
	\beta=\gamma+c_0-b_0+c_1-b_1+\frac{1}{q}-\frac{1}{p'}, \qquad \frac{1}{q}+c_0+c_1<\beta<\frac{1}{q}+c_0+c_1+n\ell,
	\end{equation*}
	and $\beta\neq 1/q+c_0+c_1+jk$, $j=1,\ldots , n-1$.
\end{theorem}
\begin{proof}First of all note that since $\sum|a_m|<\infty$, one has $|K(x,y)|\lesssim x^{b_1}y^{c_1}$ whenever $xy\leq 1$.
	
	By \eqref{generalvanishingmoments}, we can write, for any $\ell=1,\ldots ,n$,
	$$
	Ff(y)=y^{c_0+c_1}\int_0^\infty x^{b_0+b_1}f(x)\bigg(x^{-b_1}y^{-c_1}K(x,y)-\sum_{m=0}^{\ell-1} a_m (xy)^{km}\bigg)dx
	$$
	If we define 
	$$
	G_\ell(x,y)=x^{-b_1}y^{-c_1}K(x,y)-\sum_{m=0}^{\ell-1} a_m (xy)^{km}=\sum_{m=\ell}^\infty a_m (xy)^{km},
	$$
	then it is clear that for $xy\leq 1$ one has
	$$
	|G_\ell(x,y)|\leq A(xy)^{k\ell}.
	$$ 
	For $xy>1$, since $x^{-b_1}y^{-c_1}|K(x,y)|\lesssim (xy)^{c_2-c_1}$ and $c_2-c_1\leq 0$, it is also clear that $|G_\ell(x,y)|\lesssim (xy)^{k(\ell-1)}$. In conclusion,
	$$
	|G_\ell (x,y)|\lesssim \begin{cases} (xy)^{k\ell},& xy\leq 1,\\
	(xy)^{k(\ell-1)}, &xy>1,
	\end{cases}
	$$
	or equivalently,
	$$|G_\ell(x,y)|\lesssim \min\big\{(xy)^{k\ell},(xy)^{k(\ell-1)}\big\}.$$
	Hence, by Corollary~\ref{corapplications}, the transform defined as
	$$
	\mathcal{G}_\ell g(y)=y^{c_0+c_1}\int_0^\infty x^{b_0+b_1}f(x) G_\ell(x,y)\, dx
	$$
	satisfies the inequality
	$$
	\Vert y^{-\beta}\mathcal{G}_\ell g\Vert_q\lesssim \Vert x^\gamma g\Vert_p,
	$$
	provided that $\beta=\gamma+c_0-b_0+c_1-b_1+1/q-1/p'$ and
	$$
	\frac{1}{q}+c_0+c_1+k(\ell-1)<\beta<\frac{1}{q}+c_0+c_1+k\ell.
	$$
	Since the latter holds for every $\ell=1,\ldots , n$, our assertion follows.
\end{proof}

In general Theorem~\ref{thmgeneralvanishingmoments} is not true whenever $\beta=1/q+c_0+c_1+jk$ for some $j\in \{0,1,\ldots , n\}$,  as shown in the case of the Fourier transform \cite{SaWh}.

\begin{remark}In contrast with Corollary~\ref{corapplications}, in Theorem~\ref{thmgeneralvanishingmoments} we can allow $b_1=b_2$, $c_1=c_2$. This is because in order to prove Theorem~\ref{thmgeneralvanishingmoments} we apply Corollary~\ref{corapplications} to the transform $\mathcal{G}_\ell$, whose kernel satisfies $|G_\ell(x,y)|\lesssim \min\big\{(xy)^{k\ell},(xy)^{k(\ell-1)}\big\}$, thus it always satisfies the hypothesis of Corollary~\ref{corapplications}, namely $b_1=c_1=k\ell>k(\ell-1)=b_2=c_2$.
\end{remark}

A similar result to Theorem~\ref{thmgeneralvanishingmoments} also holds for the Fourier transform on $\R^n$, cf. \cite[Theorem 2]{SaWh}. See also \cite{lebrun}, where a similar problem with nonradial weights is considered.

For the Hankel transform of order $\alpha\geq -1/2$, we have the representation of $j_\alpha(xy)$ by power series \eqref{eqSeriesBessel}. Thus, on applying Theorem~\ref{thmgeneralvanishingmoments} with $b_1=c_0=c_1=0$, $b_0=2\alpha+1$ and $k=2$, we obtain the following:
\begin{corollary}\label{corhankelvanishing}
	Let $1< p\leq q<\infty$ and let $f$ be such that
	\begin{equation*}
	\label{vanishingmomenthankel}
	\int_0^\infty x^{2\alpha+1+2\ell} f(x)\, dx =0, \qquad \ell=0,\ldots , n-1, \quad n\in \N.
	\end{equation*}
	Then the inequality
	\begin{equation*}
	\label{pittvanishinghankel}
	\Vert y^{-\beta} H_\alpha f\Vert_q \leq C\Vert x^{\gamma}f\Vert_p
	\end{equation*}
	holds if $\beta=\gamma-2\alpha-1+1/q-1/p'$ and
	$$
	\frac{1}{q}<\beta < \frac{1}{q}+2n, \qquad \beta \neq \frac{1}{q}+2\ell,\, \ell=1,\ldots , n-1.
	$$
\end{corollary}
\begin{remark}
	Let us compare Theorem~\ref{SaWhthm} and Corollary~\ref{corhankelvanishing}. It is known \cite{SaWh} that if $\int_{\R} f(x)\, dx=0$, then inequality \eqref{FourierWNI} does not necessarily hold for $\beta=1+1/q$. However, it follows from  Corollary~\ref{corhankelvanishing} with $\alpha=-1/2$ (i.e., the cosine transform) that \textit{if $\int_\R f(x)\, dx=0$ and moreover $f$ is even, then inequality \eqref{FourierWNI} holds for $\beta=1+1/q$}.
\end{remark}

Let us now state a version of Theorem~\ref{thmgeneralvanishingmoments} for the sine transform. Since
$$
\sin xy= xy\sum_{m=0}^\infty \frac{(-1)^m}{(2m+1)!}(xy)^{2m}, 
$$
Theorem~\ref{thmgeneralvanishingmoments} with $b_0=c_0=0$, $b_1=c_1=1$ and $k=2$ yields the following:
\begin{corollary}
Let $1< p\leq q<\infty$ and let $f$ be such that
\begin{equation*}
\int_0^\infty x^{2\ell+1} f(x)\, dx =0, \qquad \ell=0,\ldots , n-1, \quad n\in \N.
\end{equation*}
Then the inequality
\begin{equation*}
\Vert y^{-\beta} \widehat{f}_{\sin}\Vert_q \leq C\Vert x^{\gamma}f\Vert_p
\end{equation*}
holds if $\beta=\gamma+1/q-1/p'$ and
$$
\frac{1}{q}+1<\beta < \frac{1}{q}+2n+1, \qquad \beta \neq \frac{1}{q}+2\ell+1,\, \ell=1,\ldots , n-1.
$$	
\end{corollary}

Finally, we present the statement corresponding to the $\mathscr{H}_\alpha$ transform. In view of \eqref{Hseries} and \eqref{Hest}, we apply Theorem~\ref{thmgeneralvanishingmoments} with $b_0=c_0=1/2$, $b_1=c_1=\alpha+1$ and $k=2$.
\begin{corollary}
	Let $1<p\leq q<\infty$ and $\alpha>-1/2$. Let $f$ be such that
	$$
	\int_0^\infty x^{\alpha+3/2+2 \ell}f(x)\, dx=0, \qquad \ell=0,\ldots , n-1, \quad n\in \N.
	$$
	Then the inequality
	$$
	\Vert y^{-\beta}\mathscr{H}_\alpha f\Vert_q\leq C\Vert x^{\gamma} f\Vert_p
	$$
	holds if $\beta=\gamma+1/q-1/p'$ and
	$$
	\frac{1}{q}+\alpha+\frac{3}{2}<\beta<\frac{1}{q}+\alpha+\frac{3}{2}+2n, \qquad \beta\neq \frac{1}{q}+\alpha+\frac{3}{2}+2\ell, \, \ell=1,\ldots , n-1.
	$$
\end{corollary}

\subsection{Sharpness}
To conclude this section, we show that in general Theorem~\ref{thmgeneralvanishingmoments} does not hold for $\beta=1/q+c_0+c_1$ (or equivalently, for $\gamma=1/p'+b_0+b_1$), although $\int_0^\infty x^{b_0+b_1}f(x)\, dx=0$.

\begin{proposition}\label{propsharp2}
	Let $0<q\leq\infty$ and $1<p<\infty$. Let the transform $F$ be as in \eqref{powertrans}, with kernel $K(x,y)$ of the form \eqref{kernelseries}, satisfying $|a_0|> 0$ and $\sum |a_m|=A<\infty$. Assume there is $C>0$ such that
	$$
	|K(x,y)|\leq \begin{cases}
	Cx^{b_1}y^{c_1}, & \text{if }xy\leq 1,\\
	Cx^{b_2}y^{c_2}, &\text{if }xy>1,
	\end{cases}
	$$
	where $b_j,c_j \in \R$, $j=1,2$. Furthermore, suppose  there exists $\nu\in \R$ and   $G_y^\nu(x)$ such that $(d/dx)G_y^\nu(x)= x^{\nu}K(x,y)$, and that there exists $C'>0$ for which
	\begin{equation}
	\label{primest}
	|G_y^\nu(x)|\leq C'x^{b}y^{c}, \qquad b,c\in \R, \qquad xy\geq 1,
	\end{equation}
	holds with $b-b_1-\nu<1$. Then, if $u\not\equiv 0$, the weighted norm inequality
	\begin{equation}
	\label{counterexvanishing2}
	\bigg(\int_0^\infty u(y)|Ff(y)|^q\, dy\bigg)^{1/q}\lesssim \bigg(\int_0^\infty x^{p(1/p'+b_0+b_1)}|f(x)|^p\, dx\bigg)^{1/p}
	\end{equation}
	cannot hold for all $f$ satisfying $\int_{0}^{\infty} x^{b_0+b_1}f(x)\, dx=0$.
\end{proposition}

\begin{remark}
	Note that the examples we presented above (sine, Hankel, or $\mathscr{H}_\alpha$ transforms) satisfy the hypotheses of Proposition~\ref{propsharp2}. For example, in the case of the $\mathscr{H}_\alpha$ transform ($\alpha>-1/2$), we have $b_1=\alpha+1$, $b_2=\alpha-1$, and for any $\nu\geq 1/2$, $b=\alpha+\nu$ (cf. Lemma~\ref{LemmaprimitiveStruve}).
\end{remark}
\begin{proof}[Proof of Proposition~\ref{propsharp2}]
	Define, for $N\in \N$,
	$$
	f_N(x)=\frac{1}{x^{b_0+b_1+1}}\big(\chi_{(1/N,1)}(x)-\chi_{(1,N)}(x)\big).
	$$
	Then
	$$
	\int_0^\infty x^{b_0+b_1}f_N(x)\, dx= \int_{1/N}^1 \frac{1}{x}\, dx - \int_1^N \frac{1}{x}\, dx =\log N - \log N=0,
	$$
	and
	$$
	\bigg(\int_0^\infty x^{p(1/p'+b_0+b_1)}|f_N(x)|^p\, dx\bigg)^{1/p}=\bigg(\int_{1/N}^{N}\frac{1}{x}\, dx\bigg)^{1/p} = (2\log N)^{1/p}.
	$$
	Now let $y\in (0,\infty)$ and assume $N$ is such that $1/N<1/y<N$. We have
	\begin{align*}
	y^{-c_0}|Ff(y)|&=\bigg| \int_{1/N}^1 \frac{1}{x^{b_1+1}}K(x,y)\, dx-\int_1^N \frac{1}{x^{b_1+1}}K(x,y)\, dx\bigg|\\
	&\geq \bigg| \int_{1/N}^{1/y}\frac{1}{x^{b_1+1}}K(x,y)\,dx\bigg|-2\bigg|\int_{1/y}^1 \frac{1}{x^{b_1+1}}K(x,y)\,dx \bigg|-\bigg|\int_{1/y}^N \frac{1}{x^{b_1+1}}K(x,y)\, dx\bigg|\\ &=: I_1-I_2-I_3.
	\end{align*}
	Now we proceed to estimate $I_1$ from below, and $I_2,I_3$ from above. Then, joining all such estimates and combining them with the latter inequality, we can obtain a lower estimate for $y^{-c_0}|Ff(y)|$. First,
	\begin{align*}
	I_1&=\bigg| \int_{1/N}^{1/y} \frac{K(x,y)-x^{b_1}y^{c_1}a_0+x^{b_1}y^{c_1}a_0}{x^{b_1+1}}\, dx\bigg| \\ &\geq \bigg|\int_{1/N}^{1/y} \frac{x^{b_1}y^{c_1}a_0}{x^{b_1+1}}\, dx \bigg|-\bigg| \int_{1/N}^{1/y} \frac{K(x,y)-x^{b_1}y^{c_1}a_0}{x^{b_1+1}}\, dx\bigg|.
	\end{align*}
	Since
	$$
	\bigg| \int_{1/N}^{1/y}\frac{a_0 x^{b_1}y^{c_1}}{x^{b_1+1}}\, dx\bigg| = |a_0|y^{c_1}\bigg| \int_{1/N}^{1/y} \frac{1}{x}\,dx\bigg|\geq y^{c_1}|a_0|\log N-y^{c_1}|a_0\log y|,
	$$
	and
	\begin{align*}
	\bigg| \int_{1/N}^{1/y} \frac{K(x,y)-a_0x^{b_1}y^{c_1}}{x^{b_1+1}}\, dx\bigg|\leq y^{c_1}\int_{1/N}^{1/y}x^{-1}\sum_{m=1}^\infty |a_m|(xy)^{mk}\, dx \leq Ay^{c_1+k}\int_{1/N}^{1/y}x^{k-1}\, dx \leq Ay^{c_1},
	\end{align*}
	we obtain
	$$I_1\geq y^{c_1}|a_0|\log N-y^{c_1} |a_0\log y|-Ay^{c_1}=:y^{c_1}|a_0|\log N -\eta_1(y).
	$$
	
	We now proceed to estimate $I_2$ from above. Here we distinguish two cases, namely if $1/y<1$ or $1/y\geq 1$; in the following we take $j=1$ if $1/y<1$, and $j=2$ otherwise:
	\begin{align*}
	I_2=2\bigg|\int_{1/y}^1 \frac{1}{x^{b_1+1}}K(x,y)\, dx \bigg|&\leq 2C y^{c_j}\bigg|\int_{1/y}^1 x^{b_j-b_1-1}\, dx\bigg| \leq 2Cy^{c_j}\max\{1,1/y\}\max\big\{1,y^{b_1+1-b_j}\big\}\\
	&\leq 2C y^{c_j}\max\big\{1,1/y,y^{b_1-b_j},y^{b_1+1-b_j}\big\}=:\eta_2(y).
	\end{align*}
	Finally, integration by parts and estimate \eqref{primest} yield
	\begin{align*}
	I_3&=\bigg| \int_{1/y}^N \frac{1}{x^{b_1+1+\nu}}x^\nu K(x,y)\, dx\bigg| \leq N^{-b_1-1-\nu}|G_y^\nu(N)|+y^{b_1+1+\nu}|G_y^\nu(1/y)|\\
	&\phantom{=} +(b_1+\nu+1)\int_{1/y}^N \frac{1}{x^{b_1+2+\nu}} |G_y^\nu(x)|\, dx \\
	&\leq C'y^{c} N^{b-b_1-1-\nu} +C'y^{c-b+b_1+1+\nu}+C'(b_1+1+\nu)y^{c}\int_{1/y}^N x^{b-b_1-2-\nu}\, dx\\
	&\leq C'y^{c}+C' y^{c-b+b_1+1+\nu}+C'y^{c}\bigg|\frac{b_1+1+\nu}{b-b_1-2-\nu}\bigg|(1+y^{-b+b_1+1+\nu})=:\eta_3(y).
	\end{align*}
	Thus, collecting all estimates, we obtain
	$$
	y^{-c_0-c_1}|Ff(y)|\geq |a_0|\log N -y^{-c_1}(\eta_1(y)+\eta_2(y)+\eta_3(y)).
	$$
	Since $u(y)\not\equiv 0$, we can find $0<t_1<t_2<\infty$ such that $\int_{t_1}^{t_2}u(y)\, dy>0$. Choosing $N$ so large that for every $y\in (t_1,t_2)$ there holds
	\begin{align*}
	y^{-c_0-c_1}|Ff(y)|\geq |a_0|\log N-y^{-c_1}(\eta_1(y)+\eta_2(y)+\eta_3(y))>\frac{|a_0|}{2}\log N, 
	\end{align*}
	it can be deduced from inequality \eqref{counterexvanishing2} (with the usual modification if $q=\infty$) that
	\begin{align*}
	\frac{|a_0|}{2}\log N \bigg(\int_{t_1}^{t_2} y^{q(c_0+c_1)} u(y)\, dy\bigg)^{1/q}&\leq \bigg(\int_0^\infty |Ff(y)|^q u(y)\, dy\bigg)^{1/q}\\ & \lesssim \bigg(\int_0^\infty x^{p(1/p'+b_0+b_1)}|f_N(x)|^p\, dx\bigg)^{1/p}=(2\log N)^{1/p},
	\end{align*}
	which is a contradiction, since $p>1$. 
\end{proof}

			\section{Weighted norm inequalities with general monotone functions}\label{sectionGM}
			In this section we consider the so-called \textit{general monotone} functions. We say \cite{LTnachr} that a function locally of bounded variation $f:\R_+\to \C$ is general monotone (written $f\in GM$) if there exist constants $C,\lambda>1$ such that 
			\begin{equation}
			\label{GMdef}
			\int_x^{2x} |df(t)|\leq \frac{C}{x}\int_{x/\lambda}^{\lambda x}|f(t)|\, dt, \qquad x>0,
			\end{equation}
			where $\int |df(t)|$ is understood as a Stieltjes integral. For the discrete version of general monotonicity, see \cite{TikGM}. We are interested in obtaining sufficient conditions for the weighted norm inequality
			\begin{equation}
			\label{pittgm}
			\Vert Ff\Vert_{q,u}\lesssim \Vert f\Vert_{p,v}, \qquad 1<p\leq q<\infty,
			\end{equation}
			to hold for every $f\in GM$ whenever $F$ is a transform of power-type kernel (i.e., of the form \eqref{powertrans} and satisfying \eqref{kernelestcor}).  In what follows we assume the kernel $K$ is continuous in the variable $x$. Here $u,v$ are general nonnegative weights. As a particular case, we investigate whether we can relax the sufficient conditions of Corollary~\ref{corapplications} when $u,v$ are power weights under the assumption $f\in GM$. 

			Here we assume that 
			\begin{equation}
			\label{estimategm1}
			|K(x,y)|\lesssim x^{b_1}y^{c_1},  \qquad xy<1,
			\end{equation}
			with $b_1,c_1\in \R$. Let $G(x,y)$ be such that 
			\begin{equation}
			\label{primitiveG}
			\frac{d}{dx} G(x,y)=x^{b_0} K(x,y),
			\end{equation}
			where the additive constant of $G$ is taken to be zero (such $G$ exists due to the continuity of $K$ in the variable $x$). We moreover suppose that $G(x,y)$ satisfies the estimate
			\begin{equation}
			\label{estimategm2}
			|G(x,y)|\lesssim x^{b}y^{c}, \qquad xy\geq 1,
			\end{equation}
			with $b,c\in \R$. Finally, we say that $f\in GM$ is \emph{admissible} if
			$$
			 \int_0^1 x^{b_0+b_1}|f(x)|\, dx + \int_{1}^\infty x^{b-1}|f(x)|\, dx<\infty.
			$$
			\begin{remark}\label{GMproperties}
				Let us recall some useful properties of $f\in GM$. 
				\begin{enumerate}
					\item If $\sigma\geq 0$,
				$$
				\int_{y}^\infty x^{\sigma}|df(x)|\lesssim \int_{y/\lambda}^\infty x^{\sigma-1}|f(x)|\, dx,
				$$
				where $\lambda$ is the constant from \eqref{GMdef}, see \cite[p. 111]{GLTBoas}.
				\item If $\int_1^\infty |f(x)|\, dx<\infty$, then $x|f(x)|\to 0$ as $x\to \infty$ (cf. \cite{LTnachr}).
				\item The function $x^{\sigma}f(x)$ is $GM$ for every $\sigma\in \R$.
				\end{enumerate}
			\end{remark}
			\subsection{Main results}
			First we obtain straightforward upper estimates for $G$ that follow from the upper estimates for $K$. This will provide an expression for $b,c$ in \eqref{estimategm2} in the general case.
			\begin{proposition}\label{PROPgmgeneral}
				Let $K$ satisfy \eqref{estimategm1}, and assume that $|K(x,y)|\lesssim x^{b_2}y^{c_2}$ for $xy>1$. Let $G$ be given by the relation \eqref{primitiveG}. Then,
				\begin{enumerate}[label=(\roman{*})]
				\item \label{case1} If $b_0+b_1>0$ and $b_0+b_2\neq -1$, then
				$$
				|G(x,y)|\lesssim \begin{cases}
				y^{c_1}x^{b_0+b_1+1},&\text{if }xy\leq 1,\\
				y^{c_2}x^{b_0+b_2+1}+y^{c_1-b_0-b_1-1}+y^{c_2-b_0-b_2-1}, &\text{if } xy>1.
				\end{cases}
				$$
				\item \label{case2} If $b_0+b_2<-1$ and $b_0+b_1\neq -1$, then
				$$|G(x,y)|\lesssim \begin{cases}
					y^{c_1}x^{b_0+b_1+1} +y^{c_1-b_0-b_1-1}+y^{c_2-b_0-b_2-1},&\text{if }xy\leq 1,\\
					y^{c_2}x^{b_0+b_2+1}, &\text{if } xy>1.
				\end{cases}
				$$
				\end{enumerate}
			\end{proposition}
			\begin{proof}
				\ref{case1} Since $b_0+b_1>0$, we can write $G(x,y)=\int_0^x t^{b_0}K(t,y)\, dt$, by the Fundamental Theorem of Calculus. For $x\leq 1/y$, 
				$$
				|G(x,y)|\lesssim  y^{c_1}\int_0^x t^{b_0+b_1}\, dt\lesssim y^{c_1}x^{b_0+b_1+1}, 
				$$
				whilst for $x>1/y$, using the latter estimate we obtain
				$$
				|G(x,y)|\lesssim y^{c_1-b_0-b_1-1}+\int_{1/y}^x t^{b_0}|K(t,y)|\, dt\lesssim y^{c_1-b_0-b_1-1}+y^{c_2-b_0-b_2-1}+y^{c_2}x^{b_0+b_2+1}.
				$$
				\ref{case2} Since $b_0+b_2<-1$, we can write $G(x,y)=\int_x^\infty t^{b_0}K(t,y)\, dt$, again by the Fundamental Theorem of Calculus. For $x>1/y$,
				$$
				|G(x,y)|\lesssim y^{c_2}\int_x^\infty t^{b_0+b_2}\, dt \asymp y^{c_2}x^{b_0+b_2+1}.
				$$
				For $x\leq 1/y$, using the latter estimate we obtain
				$$
				|G(x,y)|\lesssim \int_{x}^{1/y} t^{b_0}|K(t,y)|\, dt + y^{c_2-b_0-b_2-1}\lesssim y^{c_1}x^{b_0+b_1+1} +y^{c_1-b_0-b_1-1}+y^{c_2-b_0-b_2-1},
				$$
				as desired.
			\end{proof}
			\begin{remark}
			Observe that the upper estimates for $|G(x,y)|$ given in Proposition~\ref{PROPgmgeneral} are rather rough, and they are not optimal for oscillating kernels $K(x,y)$, such as $K(x,y)=j_\alpha(xy)$. However, those estimates are useful for kernels satisfying
			$$
			K(x,y)\asymp \begin{cases}
			x^{b_1}y^{c_1}, &\text{if } xy\leq 1,\\
			x^{b_2}y^{c_2}, &\text{if }xy>1.
			\end{cases}
			$$
			In fact, the Struve function $\mathbf{H}_\alpha$ with $\alpha>1/2$ satisfies the above estimate, and it can be easily checked that in this case the result given by Proposition~\ref{PROPgmgeneral} coincides with that of Lemma~\ref{LemmaprimitiveStruve}. For oscillating kernels it is more convenient to obtain these estimates by using an iterated integration by parts, as done in Lemma~\ref{LemmaprimitiveStruve} for the Struve function, or in \cite{unifconvHankel} for the Bessel function (in both cases the estimates are sharp).
			\end{remark}
		
			The following lemma yields an upper pointwise estimate for $Ff$.			\begin{lemma}\label{lemmagmest}
				Let $f\in GM$ be an admissible function. Assume \eqref{estimategm1} holds, and $G(x,y)$ defined by \eqref{primitiveG} satisfies \eqref{estimategm2} with  $b\geq 0$. Then the transform
				$$
				Ff(y)=y^{c_0}\int_0^\infty x^{b_0}f(x)K(x,y)\, dx
				$$
				satisfies the pointwise estimate
				\begin{equation}
				\label{Ffgmest}
				|Ff(y)|\lesssim y^{c_0+c_1}\int_0^{1/y}|f(x)|x^{b_0+b_1}\, dx + y^{c+c_0}\int_{1/(\lambda y)}^\infty x^{b-1}|f(x)|\, dx,
				\end{equation}				
				where $\lambda$ is the constant from \eqref{GMdef}.
			\end{lemma}
			Note that if $f\in GM$ is admissible, it follows from Lemma~\ref{lemmagmest} that $Ff(y)$ is defined in $(0,\infty)$.
			\begin{proof}[Proof of Lemma~\ref{lemmagmest}]			
				In view of \eqref{estimategm1}, we have
				$$
				|Ff(y)|\lesssim y^{c_0+c_1} \int_{0}^{1/y} x^{b_0+b_1}|f(x)|\, dx + y^{c_0}\bigg| \int_{1/y}^\infty f(x) x^{b_0} K(x,y)\, dx   \bigg| =:I_1+|I_2|.
				$$
				Partial integration on $I_2$ yields the estimate
				$$
				|I_2| \leq y^{c_0} |f(x)G(x,y)|\bigg|_{1/y}^\infty +y^{c_0} \int_{1/y}^\infty|G(x,y)\, df(x)|.
				$$
				First, since $f$ is admissible, it follows from \eqref{estimategm2} and 2. and 3. of Remark~\ref{GMproperties} that
				$$
				\lim_{x\to \infty}|f(x)G(x,y)|\lesssim y^c\lim_{x\to \infty} x^{b}|f(x)|=0.
				$$
				Secondly, since $b\geq 0$, we deduce from 1. of Remark~\ref{GMproperties} and \eqref{estimategm2} that
				$$
				y^{c_0}|f(1/y)G(1/y,y)|\lesssim y^{c+c_0-b}|f(1/y)|\leq y^{c+c_0-b}\int_{1/y}^\infty |df(x)|\lesssim y^{c+c_0}\int_{1/(\lambda y)}^\infty x^{b-1}|f(x)|\, dx.
				$$
				Finally, similarly as above,
				\begin{align*}
				y^{c_0}\int_{1/y}^\infty |G(x,y)\, df(x)|&\lesssim y^{c+c_0}\int_{1/y}^\infty x^{b}|df(x)|\lesssim y^{c+c_0}\int_{1/(\lambda y)}^\infty x^{b-1}|f(x)|\, dx,
				\end{align*} 
				and therefore \eqref{Ffgmest} is established.
				\end{proof}
				\begin{remark}\label{rmkgm}
					Note that if $b-c-1=b_0+b_1-c_1$, one may take $\lambda=1$ in \eqref{Ffgmest}, since
					$$ 
					y^{c+c_0}\int_{1/(\lambda y)}^{1/y} x^{b-1}|f(x)|\, dx \asymp  y^{c_0+c_1}\int_{1/(\lambda y)}^{1/y}x^{b_0+b_1}|f(x)|\, dx \leq y^{c_0+c_1}\int_0^{1/y} x^{b_0+b_1}|f(x)|\, dx
					$$
				\end{remark}
We are now in a position to prove sufficient conditions for the inequality \eqref{generalweightednorm} to hold.
\begin{theorem}\label{thmgeneralGM}
	Let $1<p\leq q<\infty$ and $f\in GM$ be admissible. Let $F$ be as in \eqref{powertrans}. Assume \eqref{estimategm1} holds, and $G(x,y)$ defined by \eqref{primitiveG} satisfies \eqref{estimategm2} with  $b\geq 0$. Then,  inequality \eqref{pittgm} holds provided that there exists $C>0$ such that for every $r>0$,
	\begin{align}
	\bigg(\int_0^{1/r} u(y)y^{(c_0+c_1)q}\, dy\bigg)^{1/q}\bigg(\int_0^r v(x)^{1-p'}x^{(b_0+b_1)p'}\, dx\bigg)^{1/p'}&\leq C, \label{PittGMcond1}\\
	\bigg(\int_{1/(\lambda r)}^\infty u(y) y^{(c+c_0)q}\, dy\bigg)^{1/q}\bigg(\int_{r}^\infty v(x)^{1-p'} x^{(b-1)p'}\, dx\bigg)^{1/p'}&\leq C, \label{PittGMcond2}
	\end{align}
	where $\lambda$ is the constant from \eqref{GMdef}.
\end{theorem}
\begin{remark}
	Note that under certain assumptions on the parameters $c$, $b$, $c_i$, and $b_i$ ($i=0,1$), we can use the gluing lemma (Lemma~\ref{lemmaglue}) to rewrite conditions \eqref{PittGMcond1} and \eqref{PittGMcond2} as one single condition, similarly as done with Theorem~\ref{suffFtransforms} and Corollary~\ref{CORgluing}.
\end{remark}
\begin{proof}[Proof of Theorem~\ref{thmgeneralGM}]
	Using the estimate \eqref{Ffgmest}, we can write
	\begin{align*}
	\bigg(\int_0^\infty u(y)|Ff(y)|^q\, dy\bigg)^{1/y}&\lesssim \bigg(\int_0^\infty u(y)\bigg(y^{c_0+c_1}\int_0^{1/y} x^{b_0+b_1}|f(x)|\, dx\bigg)^{q}dy\bigg)^{1/q} \\
	&\phantom{=}+\bigg(\int_0^\infty u(y)\bigg(y^{c+c_0}\int_{1/(\lambda y)}^\infty x^{b-1}|f(x)|\, dx\bigg)^qdy\bigg)^{1/q}=:I_1+I_2.
	\end{align*}
	On the one hand, by Lemma~\ref{hardylema} and the change of variables $y\to 1/y$, the inequality
	\begin{align*}
	I_1&=\bigg(\int_0^\infty\frac{u(1/y)}{y^{2+(c_0+c_1)q}}\bigg(\int_{0}^y x^{b_0+b_1}|f(x)|\, dx\bigg)^q dy\bigg)^{1/q}\lesssim \bigg(\int_0^\infty v(x) |f(x)|^p\, dx\bigg)^{1/p} =\Vert f\Vert_{p,v},
	\end{align*}
	holds if
	$$
	\bigg(\int_r^\infty \frac{u(1/y)}{y^{2+(c_0+c_1) q}} \, dy\bigg)^{1/q}\bigg(\int_{0}^r v(x)^{1-p'}x^{(b_0+b_1)p'}\, dx\bigg)^{1/p'}\leq C, \qquad r>0,
	$$
	or equivalently, if \eqref{PittGMcond1} is satisfied. On the other hand, again by Lemma~\ref{hardylema} and the change of variables $y\to 1/y$, the inequality
	$$
	I_2\asymp \bigg(\int_r^\infty \frac{u\big( (\lambda y)^{-1}\big)}{y^{2+(c+c_0) q}} \bigg(\int_y^\infty x^{b-1}|f(x)|\, dx\bigg)^q dy\bigg)^{1/q} \lesssim \bigg(\int_0^\infty  v(x) |f(x)|^p\, dx \bigg)^{1/p}=\Vert f\Vert_{p,v},
	$$
	holds provided that
	$$
	\bigg(\int_{0}^r \frac{u\big( (\lambda y)^{-1}\big)}{y^{2+(c+c_0)q}}\, dy\bigg)^{1/q}\bigg(\int_r^\infty v(x)^{1-p'}x^{(b-1)p'}\, dx\bigg)^{1/p'}\leq C,\qquad r>0,
	$$
	or equivalently, if \eqref{PittGMcond2} holds.	
\end{proof}
Let us also state sufficient conditions for inequality \eqref{pittgm} whenever $u$ and $v$ are power weights.
\begin{corollary}\label{corGMpower}
	Let $1<p\leq q<\infty$ and $f\in GM$ be admissible. Let $F$ be as in \eqref{powertrans}. Assume \eqref{estimategm1} holds, and $G(x,y)$ defined by \eqref{primitiveG} satisfies \eqref{estimategm2} with  $b\geq 0$ and $c<c_1$. Then, inequality \eqref{generalweightednorm} holds with
	$$
	\beta=\gamma+c_0-b_0+c_1-b_1+\frac{1}{q}-\frac{1}{p'}, \qquad \frac{1}{q}+c_0+c<\beta<\frac{1}{q}+c_0+c_1.
	$$
\end{corollary}
The proof of the latter is omitted, as it is essentially an application of Theorem~\ref{thmgeneralGM} with $u(y)=y^{-\beta q}$, $v(x)=x^{\gamma p}$. It then follows the same steps as the proof of Theorem~\ref{theoremsuffpowerweights}.
\begin{remark}\label{rmkGM}
	Let us compare the conditions for $\beta$ in Corollaries~\ref{corapplications} and \ref{corGMpower}. On the one hand, we observe that in both statements the condition $\beta<1/q+c_0+c_1$ is required. On the other hand, Corollary~\ref{corapplications} requires that $\beta>1/q+c_0+c_2$, whilst Corollary~\ref{corGMpower} requires $\beta>1/q+c_0+c$. Therefore, in order for Corollary~\ref{corGMpower} to yield a nontrivial result we need to assume $c<c_2$.
\end{remark}

\subsection{Examples}
Let us present sufficient conditions for inequality \eqref{generalweightednorm} to hold for the above transforms under the assumption $f\in GM$. Some of the following results are already known, some others are new. It is worth noting that in all examples we show below, the conditions on the parameters $b\geq 0$, $c<c_1$, and $b-c-1=b_0+b_1-c_1$ hold (in fact, the latter condition can be omitted when $u,v$ are power weights, see \eqref{PittGMcond2}).
	\begin{enumerate}
		\item For the Fourier transform (in this case the integration in \eqref{transform} is performed in the interval $(-\infty,\infty)$, but it can be divided into two integrals over the interval $(0,\infty)$), since $K(x,y)=e^{ixy}$ and $G(x,y)=(iy)^{-1}e^{ixy}$, we have $b=b_0=b_1=c_0=c_1=0$ and $c=-1$. Thus, for $f\in GM$, the sufficient conditions that guarantee the inequality
		$$
		\Vert y^{-\beta} \widehat{f}\Vert_q\lesssim \Vert x^\gamma f\Vert_p
		$$
		are $\beta=\gamma+1/q-1/p'$ and $-1+1/q<\beta<1/q$. For the cosine transform the situation is similar, i.e., the sufficient conditions are the same, and in both cases those are also necessary (cf. \cite{GLTBoas,LTparis}).
		\item The sine transform (for which $K(x,y)=\sin xy$ and $G(x,y)=-y^{-1}\cos xy$) satisfies $b_1=c_1=1$, $b=b_0=c_0=0$ and $c=-1$, thus, if $f\in GM$, the sufficient conditions for the inequality
		\begin{equation*}\label{sinepittGM}
		\Vert y^{-\beta}\widehat{f}_{\sin} \Vert_q\lesssim \Vert x^\gamma f\Vert_p
		\end{equation*}
		to hold are $\beta=\gamma+1/q-1/p'$ and $-1+1/q<\beta<1+1/q$. These conditions are also necessary, as shown in \cite{LTparis}. 
		\item The classical Hankel transform of order $\alpha\geq -1/2$ \eqref{defHankel}  has kernel $K(x,y)=j_\alpha(xy)$ satisfying $j_\alpha(xy)\asymp 1$ for $xy\leq 1$, and moreover
		$$
		|G(x,y)|\lesssim y^{-\alpha-3/2} x^{\alpha+1/2}, \qquad x,y\in \R_+,
		$$
		cf. \cite[Lemma 2.6]{unifconvHankel}. Thus, on applying Corollary~\ref{corGMpower} with $b_0=2\alpha+1$, $b_1=c=c_1=0$, $b=\alpha+1/2$ and $c=-\alpha-3/2$, we get that the inequality
		$$
		\Vert y^{-\beta} H_\alpha f\Vert_q\lesssim \Vert x^{\gamma}f\Vert_p
		$$
		holds with $\beta=\gamma-2\alpha-1+1/q-1/p'$ and $1/q-\alpha-3/2<\beta<1/q$. These sufficient conditions are also necessary, as proved in \cite{DCGT}. This includes the cosine transform ($\alpha=-1/2$), (see also \cite{LTparis}).
		
		\item The $\mathscr{H}_\alpha$ transform with $\alpha>-1/2$ \eqref{defHtrans} has kernel $K(x,y)=\mathbf{H}_\alpha(xy)$ satisfying $\mathbf{H}_\alpha(xy)\asymp (xy)^{\alpha+1}$ for $xy\leq 1$. By Lemma~\ref{LemmaprimitiveStruve}, we have
		$$
		|G(x,y)|\lesssim y^{\alpha-1}x^{\alpha+1/2}, \qquad xy>1.
		$$ Hence, applying Corollary~\ref{corGMpower} with $b_0=c_0=1/2$, $b_1=c_1=\alpha+1$, $b=\alpha+1/2$ and $c=\alpha-1$, we get that the inequality
		\begin{equation}
		\label{wniHtrans}
		\Vert y^{-\beta} \mathscr{H}_\alpha f \Vert_q\lesssim \Vert x^{\gamma}f\Vert_p
		\end{equation}
		holds with $\beta=\gamma+1/q-1/p'$ and $1/q+\alpha-1/2<\beta<1/q+\alpha+3/2$. Notice that for $\alpha\geq 1/2$, this yields no improvement with respect to the general case (cf. \eqref{suffHtrans}), but for $-1/2<\alpha<1/2$, the $GM$ hypothesis on $f$ allows us to drop the condition $\beta \geq \max\{1/q-1/p',0\}$.
	\end{enumerate}

To conclude, we prove that the range of $\beta$ for which \eqref{wniHtrans} holds given by Corollary~\ref{corGMpower} is sharp. 
\begin{theorem}
	Let $1<p\leq q<\infty$ and $f\in GM$ be admissible. Inequality \eqref{wniHtrans} holds if and only if
	\begin{equation}
	\label{wniHtranscond}
	\beta=\gamma+1/q-1/p', \qquad 1/q+\alpha-1/2<\beta<1/q+\alpha+3/2.
	\end{equation}
\end{theorem}
\begin{proof}
We only need to prove that if \eqref{wniHtrans} holds for every admissible $f\in GM$, then \eqref{wniHtranscond} holds. For $\alpha>-1/2$ and $r>0$, consider the function $f_r(x)=x^{\alpha+1/2}\chi_{(0,r)}(x)$. Note that $f\in GM$ and it is admissible. By \cite[\S 11.2 (2)]{EMOTtables}, one has
$$
\mathscr{H}_\alpha f_r(y)=r^{\alpha+1}y^{-1/2}\mathbf{H}_{\alpha+1}(ry).
$$
On the one hand
$$
\Vert x^\gamma f_r\Vert_p=\bigg(\int_0^r x^{p(\gamma+\alpha+1/2)}\, dx\bigg)^{1/p} \asymp r^{\gamma+\alpha+1/2+1/p}, 
$$
provided that $\gamma+\alpha+1/2>-1/p$. On the other hand, 
$$
\Vert y^{-\beta}\mathscr{H}_\alpha f_r \Vert_q =r^{\alpha+1}\bigg( \int_0^\infty y^{-q(\beta+1/2)} |\mathbf{H}_{\alpha+1}(ry)|^q\, dy\bigg)^{1/q}.
$$
Since $\mathbf{H}_{\alpha+1}(ry)\asymp (ry)^{\alpha+2}$ whenever $ry\leq 1$, the latter integral is convergent near the origin if and only if $\beta<1/q+\alpha+3/2$, whereas since $\mathbf{H}_{\alpha+1}(ry)\asymp (ry)^\alpha$ whenever $ry$ is large enough (cf. Remark~\ref{remarkHbigalpha}), the integral converges near infinity if and only if $\beta>1/q+\alpha-1/2$. In order to conclude the proof, we note that
$$
\Vert y^{-\beta}\mathscr{H}_\alpha f_r\Vert_q\geq r^{2\alpha+3}\bigg(\int_0^{1/r}y^{q(-\beta+\alpha+3/2)}\, dy \bigg)^{1/q} \asymp r^{\alpha+3/2+\beta-1/q}.
$$
Combining the latter with inequality \eqref{wniHtrans} and the equivalence $\Vert x^\gamma f\Vert_p \asymp r^{\gamma+\alpha+1/2+1/p}$, we get that $r^{\alpha+3/2+\beta-1/q}\lesssim r^{\gamma+\alpha+1/2+1/p}$ for every $r>0$, i.e., $\beta=\gamma+1/q-1/p'$.
\end{proof}

\noindent \textbf{Acknowledgements}. The author acknowledges the support of Fundaci\'o Ferran Sunyer i Balaguer from Institut d'Estudis Catalans during the carrying out of this work.

Alberto Debernardi,\\
 Centre de Recerca Matem\`{a}tica and Universitat Aut\`onoma de Barce\-lona, \\ Campus de Bellaterra, Edifici C 08193 Bellaterra (Barcelona).\\
E-mail: \texttt{adebernardipinos@gmail.com}


\begin{thebibliography}{11}
	
\bibitem{AH}
N. E. Aguilera and E. O. Harboure,
\newblock \textit{On the search for weighted norm inequalities for the {F}ourier transform},
\newblock Pacific J. Math. \textbf{104} (1983), 1--14.	
		
\bibitem{babenko}
K. I. Babenko,
\newblock \textit{An inequality in the theory of {F}ourier integrals} (Russian),
\newblock Izv. Akad. Nauk SSSR Ser. Mat. \textbf{25} (1961), 531--542.		
		
\bibitem{BeAnnals}
W. Beckner,
\newblock \textit{Inequalities in {F}ourier analysis},
\newblock Ann. of Math. (2) \textbf{102} (1975), 159--182.		
		
\bibitem{Be}
W. Beckner,
\newblock \textit{Pitt's inequality and the uncertainty principle},
\newblock Proc. Amer. Math. Soc. \textbf{123} (6) (1995), 1897--1905.

\bibitem{BHmeasure}
J. J. Benedetto and H. P. Heinig,
\textit{Fourier transform inequalities with measure weights},
Adv. Math. \textbf{96} (2) (1992), 194--225.

\bibitem{BH}
J. J. Benedetto and H. P. Heinig,
\newblock \textit{Weighted Fourier Inequalities: New Proofs and Generalizations},		
\newblock J. Fourier Anal. Appl. \textbf{9} (2003), 1--37.

\bibitem{BSbook}
C. Bennett and R. Sharpley,
\newblock \textit{Interpolation of Operators},
\newblock Academic Press, Inc., Boston, 1988.

\bibitem{BloomLaplace}
S. Bloom,
\newblock \textit{Hardy integral estimates for the {L}aplace transform},
\newblock Proc. Amer. Math. Soc. \textbf{116} (2) (1992), 417--426.

\bibitem{BS}
S. Bloom and G. Sampson,
\newblock \textit{Weighted spherical restriction theorems for the {F}ourier transform},
\newblock Illinois J. Math. \textbf{36} (1992), 73--101.

\bibitem{BozaSoria}
S. Boza and J. Soria,
\newblock \textit{Weak-type boundedness of the Fourier transform on rearrangement invariant function spaces},
\newblock Proc. Edinb. Math. Soc. (2) \textbf{61} (3) (2018), 879--890.

\bibitem{hardyineq}
J. S. Bradley,
\newblock \textit{Hardy inequalities with mixed norms},
\newblock Canad. Math. Bull. \textbf{21} (4) (1978), 405--408.

\bibitem{calderon}
A. P. Calder\'on,
\newblock \textit{Spaces between {$L^{1}$} and {$L^{\infty }$} and the	theorem of {M}arcinkiewicz},
\newblock Studia Math. \textbf{26} (1966), 273--299.

\bibitem{CRSlorentz}
M. J. Carro, J. A. Raposo, and J. Soria,
\newblock \textit{Recent developments in the theory of {L}orentz spaces and	weighted inequalities},
\newblock Mem. Amer. Math. Soc. \textbf{187} (2007).

\bibitem{DC} L. De Carli,
\newblock \textit{On the {$L^p$}-{$L^q$} norm of the {H}ankel transform and related operators},
\newblock J. Math. Anal. Appl. \textbf{348} (2008),  366--382.

\bibitem{DCGT}
L. De Carli, D. Gorbachev, and S. Tikhonov, 
\newblock \textit{Pitt and Boas inequalities for Fourier and Hankel transforms},
J. Math. Anal. Appl. \textbf{408} (2) (2013), 762--774.

\bibitem{DCGTrestr}
L. De Carli, D. Gorbachev, and S. Tikhonov, 
\newblock \textit{Pitt inequalities and restriction theorems for the Fourier transform},
\newblock Rev. Mat. Iberoam. \textbf{33} (3) (2017), 789--808.
	
\bibitem{lebrun}
C. Carton-Lebrun,
\newblock \textit{Fourier inequalities with nonradial weights},
\newblock Trans. Amer. Math. Soc. \textbf{333} (2) (1992), 751--767.
	
\bibitem{unifconvHankel}
A. Debernardi, 
\newblock \textit{Uniform convergence of Hankel transforms}, 
\newblock J. Math. Anal. Appl. \textbf{468} (2) (2018), 1179--1206.
	
\bibitem{EMOT}
A. Erd{\'e}lyi, W. Magnus, F. Oberhettinger, and F. G. Tricomi,
\newblock \textit{Higher Transcendental Functions},
\newblock McGraw-Hill, New York, 1953.

\bibitem{EMOTtables}
A. Erd{\'e}lyi, W. Magnus, F. Oberhettinger, and F. G. Tricomi,
\newblock \textit{Tables of Integral Transforms},
\newblock McGraw-Hill, New York, 1954.
	
\bibitem{Fef}
C. Fefferman,
\newblock \textit{Inequalities for strongly singular convolution operators},
\newblock Acta Math. \textbf{124} (1970), 9--36.

\bibitem{GKP}
A. Gogatishvili, A. Kufner, and L.-E. Persson,
\newblock \textit{Some new scales of weight characterizations of the class {$B_p$}},
\newblock Acta Math. Hungar. \textbf{123} (4) (2009), 365--377.	
	
\bibitem{GLTBoas}
D. Gorbachev, E. Liflyand, and S. Tikhonov,
\newblock \textit{Weighted {F}ourier inequalities: {B}oas' conjecture in {$\mathbb{R}^n$}},
\newblock J. Anal. Math. \textbf{114} (2011), 99--120.
	
\bibitem{DLT} D. Gorbachev, E. Liflyand, and S. Tikhonov,
\newblock \textit{Weighted norm inequalities for integral transforms},
\newblock Indiana Univ. Math. J.  \textbf{67} (2018), 1949--2003.

\bibitem{HaTi}
G. H. Hardy and E. C. Titchmarsh,
\newblock \textit{A class of Fourier kernels},
\newblock Proc. London Math. Soc. \textbf{S2-35} (1933), 116--155.

\bibitem{heinigWeighted}
H. P. Heinig,
\newblock \textit{Weighted norm inequalities for classes of operators},
\newblock Indiana Univ. Math. J. \textbf{33} (4) (1984), 573--582.

\bibitem{JurkatSampson}
W. B. Jurkat and G. Sampson
\newblock \textit{On rearrangement and weight inequalities for the {F}ourier transform},
\newblock Indiana Univ. Math. J. \textbf{33} (2) (1984), 257--270.

\bibitem{jacobi}
T. H. Koornwinder,
\newblock \textit{Jacobi functions and analysis on noncompact semisimple {L}ie groups}, in Special functions: group theoretical aspects and applications, Reidel, Dodrecht, 1984, 1--85.

\bibitem{LTnachr}
E. Liflyand and S. Tikhonov,
\newblock \textit{A concept of general monotonicity and applications},
\newblock Math. Nachr. \textbf{284} (8-9) (2011), 1083--1098.

\bibitem{LTparis}
E. Liflyand and S. Tikhonov,
\newblock \textit{Extended solution of {B}oas' conjecture on {F}ourier transforms},
\newblock C. R. Math. Acad. Sci. Paris \textbf{346} (21-22) (2008), 1137--1142.

\bibitem{Lorentzspaces}
G. G. Lorentz,
\newblock \textit{On the theory of spaces {$\Lambda$}},
\newblock Pacific J. Math. \textbf{1} (1951), 411--429.

\bibitem{MuckenhouptNote}
B. Muckenhoupt,
\newblock \textit{A note on two weight function conditions for a {F}ourier transform norm inequality},
\newblock Proc. Amer. Math. Soc. \textbf{88} (1983), 97--100.	

\bibitem{MuckenhouptWeighted}
B. Muckenhoupt,
\newblock \textit{Weighted norm inequalities for the {F}ourier transform},
\newblock Trans. Amer. Math. Soc. \textbf{276} (2) (1983), 729--742.

\bibitem{NT}
E. Nursultanov and S. Tikhonov,
\newblock \textit{Net spaces and boundedness of integral operators},
\newblock J. Geom. Anal. \textbf{21} (4) (2011), 950--981.

\bibitem{oinarov}
R. O\u\i narov,
\newblock \textit{Two-sided estimates for the norm of some classes of integral	operators},
\newblock Trudy Mat. Inst. Steklov. \textbf{204} (1993), 240--250 (translation in Proc. Steklov Inst. Math. \textbf{204} (3) (1994), 205-–214).

\bibitem{RonCan}
P. G. Rooney,
\newblock \textit{On the {${\mathscr Y}_{\nu }$}\ and {${\mathscr H}_{\nu }$}\
	transformations},
\newblock Canad. J. Math. \textbf{32} (5) (1980), 1021--1044.

\bibitem{SaWh}
C. Sadosky and R. L. Wheeden,
\newblock \textit{Some weighted norm inequalities for the {F}ourier transform of
	functions with vanishing moments},
\newblock Trans. Amer. Math. Soc. \textbf{300} (2) (1987), 521--533.

\bibitem{sinnamonlorentz}
G. Sinnamon,
\newblock \textit{The {F}ourier transform in weighted {L}orentz spaces},
\newblock Publ. Mat. \textbf{47} (2003), 3--29.

\bibitem{SWfourier}
E. M. Stein and G. Weiss,
\newblock \textit{Introduction to {F}ourier analysis on {E}uclidean spaces},
\newblock Princeton University Press, Princeton, N.J., 1971.

\bibitem{TikGM}
S. Tikhonov,
\newblock \textit{Trigonometric series with general monotone coefficients},
J. Math. Anal. Appl. \textbf{326} (2007), 721--735.

\bibitem{titchmarsh}
E. C. Titchmarsh,
\newblock \textit{Introduction to the Theory of Fourier Integrals},
\newblock Clarendon Press, Oxford, 1937.

\bibitem{tomas}
P. A. Tomas,
\newblock \textit{A restriction theorem for the {F}ourier transform},
\newblock Bull. Amer. Math. Soc. \textbf{81} (1975), 477--478.

\bibitem{watBesselfn}
G. N. Watson,
\newblock \textit{A Treatise on the Theory of Bessel Functions},
\newblock Cambridge University Press, Cambridge, 1995.

\bibitem{wat}
G. N. Watson,
\newblock \textit{General transforms},
\newblock Proc. London Math. Soc. \textbf{S2-35} (1933), 156--199.
\end{thebibliography}
\end{document}